\documentclass[a4paper,11pt]{amsart}
\usepackage[dvips]{graphicx}
\usepackage{geometry}
\usepackage{float}
\usepackage{amsmath, amscd, amsthm, amssymb, mathrsfs, ascmac}
\usepackage{footmisc}
 \everymath{\displaystyle}
\usepackage[all]{xy}
\usepackage{epstopdf}
\theoremstyle{definition}
\makeatletter
\newcommand{\Z}{\mathbb{Z}}
\newcommand{\Q}{\mathbb{Q}}
\newcommand{\C}{\mathbb{C}}
\newcommand{\R}{\mathbb{R}}

\newcommand{\G}{\mathbb{G}}
\newcommand{\X}{\mathbb{X}}

\renewcommand\@makefntext{\parindent=1zw
 \leftskip=2zw 
 \noindent \llap{\hss \@makefnmark } }
\makeatother
\newtheorem{Def}{Definition}[section]
\newtheorem{Thm}[Def]{Theorem}
\newtheorem{Prop}[Def]{Proposition}
\newtheorem{Lem}[Def]{Lemma}
\newtheorem{Cor}[Def]{Corollary}

\newtheorem{Rmk}[Def]{Remark}

\newtheorem*{Main Theorem}{Main Theorem 1}
\title[$p$-adic Hecke $L$-functions and $p$-adic Eisenstein-Kronecker Series]
{$p$-adic Eisenstein-Kronecker Series and Non-Critical Values of $p$-adic Hecke $L$-function
of an Imaginary Quadratic Field when the Conductor is Divisible by $p$}
\author[Tomoki Hirotsune]{Tomoki Hirotsune}
\address{Department of Mathematics, Keio University, 3-14-1 Kouhoku-ku, Hiyoshi, Yokohama, 223-8522 Japan}
\email{bibun.to.sekibun@gmail.com}
\date{}                                           
\begin{document}
\begin{abstract}
We relate non-critical special values $p$-adic $L$-functions associated to algebraic Hecke characters of an imaginary quadratic number field with class number one
to $p$-adic Coleman function called the $p$-adic Eisenstein-Kronecker series, when the conductors of the algebraic Hecke characters are divisible by $p$.
\end{abstract}

\maketitle
\tableofcontents

\section{Introduction}
Let $p$ be a rational prime number. 
The purpose of this article is to relate non-critical values of the $p$-adic $L$-functions associated to algebraic Hecke characters whose conductors are divisible by $p$ 
of an imaginary quadratic field with class number 1 to $p$-adic Eisenstein-Kronecker series.
We expect to use this result in the future to consider the $p$-adic Beilinson conjecture for the corresponding Hecke character.
The Beilinson conjectures about special values of $L$-functions are a vast generalization of the class number formula for Dedekind zeta function (see \cite{Beilinson}),
which state that non-critical values of $L$-functions of an algebraic variety
can be expressed using invariants arising from the Beilinson regulator map.
More generally, these conjectures can be formulated for motives.
The $p$-adic Beilinson conjectures are the $p$-adic analogues of the Beilinson conjectures, which 
state that non-critical values of $p$-adic $L$-functions of an algebraic variety may be concretely expressed by invariants arising from the syntomic regulator.
The $p$-adic Beilinson conjectures were formulated and proved by Gros in the case of Dirichlet motives (\cite{Gros1}, \cite{Gros2})
and were generalized to $p$-adic $L$-functions of motives by Perrin-Riou (see \cite{Perrin} \S4.2).
For $p$-adic $L$-functions of Abelian Artin motives, Coleman related special values of these $p$-adic $L$-functions to $p$-adic polylogarithms defined
by using his theory of $p$-adic integration (\cite{Col1}). 
The polylogarithms have a motivic interpreter defined by Beilinson and Deligne.  
In other words, the Beilinson conjectures and their $p$-adic analogues suggest that special values of $L$-functions and $p$-adic $L$-functions may be expressed by using
motivic elements.

The $p$-adic $L$-functions of algebraic Hecke characters of our case was first constructed by Manin and Vi\v{s}ik
\cite{MaVi}, N. Katz \cite{Katz2}, R. I. Yager \cite{Yag}, and de Shalit \cite{dS}.
Bannai and Kobayashi gave a different construction of these $p$-adic $L$-functions using the 
Kronecker theta functions associated to the Poincar\'e bundle  (see \cite{BKT} \S2). 
In addition, Bannai, Kobayashi, and Tsuji expressed the elliptic polylogarithms on an elliptic curve with complex multiplication
in terms of the ``Eisenstein-Kronecker series" (for details, see \cite{Zag}), 
by using the Kronecker theta functions (see \cite{BK} Theorem 1.17).
Deuring's theorem indicates that $L$-functions of an elliptic curve with complex multiplication by the integer ring of an imaginary quadratic field 
can be expressed as Hecke $L$-functions of an imaginary quadratic field.
Then the $p$-adic $L$-functions of an elliptic curve with complex multiplication become the $p$-adic $L$-functions of an imaginary quadratic field
associated to algebraic Hecke characters.

When the conductors of algebraic Hecke characters are not divisible by $p$, Bannai, Kobayashi, and Tsuji related non-critical values of
$p$-adic $L$-functions of an imaginary quadratic field with class number 1
associated to algebraic Hecke characters to $p$-adic Eisenstein-Kronecker series by the measure construction with the connection functions (see \cite{BKT} Proposition 2.27).
In addition, they calculated  the syntomic regulator in terms of $p$-adic Eisenstein-Kronecker series (see \cite{BKT} \S4.3) in order to construct and explicitly
calculate the $p$-adic elliptic polylogarithm,
and expressed concretely non-critical values of $p$-adic $L$-functions associated to algebraic Hecke characters whose conductors are not divisible by $p$
with $p$-adic elliptic polylogarithm class 
in the rigid syntomic cohomology (see \cite{BKT} \S5.3).  In our paper, we extend the result (\cite{BKT} Proposition 2.27) of Bannai, Kobayashi, and Tsuji 
to Hecke characters whose conductors are divisible by $p$.
In future research, we hope to apply our result to consider the $p$-adic Beilinson conjecture for algebraic Hecke characters, extending the work of \cite{BKi}
to the case for characters whose conductors are divisible by $p$.

Let $K$ be an imaginary quadratic field and $E$ be an elliptic curve defined over $K$ with complex multiplication
by the integer ring $\mathcal{O}_{K}$ of $K$. 
Then $K$ has class number 1 since the class number of $K$ equals to $[K(j(E)):K]$
for $j$-invariant $j(E)$ of $E$ by (\cite{Sil2} II \S4 Theorem 4.3).
We fix a Weierstrass model $E$ by
$$
E: y^2=4x^3-g_{2}x-g_{3}\quad (g_2, g_3\in \mathcal{O}_{K})
$$
defined over $\mathcal{O}_{K}$. 
We assume that $E$ has a good ordinary reduction at a prime above $p\geq 5$.
Let $\Gamma$ be the period lattice corresponding to the invariant differential $\omega=dx/y$ obtained by the uniformization theorem.
Then we have $E(\mathbb{C})\cong \mathbb{C}/\Gamma$.  By (\cite{Sil1} VI Theorem 4.1(a)) we have
$\operatorname{End}(E(\mathbb{C}))\cong \{\alpha\in \mathbb{C}|\  \alpha\Gamma\subset \Gamma \}$. 
In addition, since the elliptic curve $E$ has complex multiplication by $\mathcal{O}_{K}$, we have
$\mathcal{O}_{K}\Gamma=\Gamma$.
Hence there exists \textbf{a complex period} $\Omega\in \mathbb{C}^{\times}$ satisfying $\Gamma=\Omega \mathcal{O}_{K}$.
Let $\psi:=\psi_{E/K}: \mathbb{A}_{K}^{\times}/K^{\times}\to \mathbb{C}^{\times}$
be the Hecke character of $K$ associated to $E$, 
where $\mathbb{A}_{K}^{\times}$ is the id\`ele group of  $K$ (\cite{Sil2} II Theorem 9.2).
Let $\mathfrak{p}$ be a prime ideal of $\mathcal{O}_{K}$ above $p$.
Since $E$ has a good reduction at a prime above $p$, by (\cite{Sil2} II Theorem 9.2), 
the Hecke character $\psi$ of $K$ associated to $E$ is unramified, i.e. $\psi(\mathcal{O}_{K_{\mathfrak{p}}}^{\times})=1$,
where $K_{\mathfrak{p}}$ is the completion at $\mathfrak{p}$ and $\mathcal{O}_{K_{\mathfrak{p}}}$ is its integer ring.
Then we define $\psi(\mathfrak{p})$ to be 
$$
\psi(\mathfrak{p}):=\psi(...,1,1,\underset{\substack{\uparrow \\ \mathfrak{p}-\text{th}}}{\pi},1,1,...).
$$ 
where $\pi$ is a uniformizer at $\mathfrak{p}$. Since $\psi$ is unramified at $\mathfrak{p}$, $\psi(\mathfrak{p})$ is well-defined independent of the choice of $\mathfrak{p}$.
From the assumption that $E$ is ordinary at a prime above $p$, 
$p$ splits $(p)=\mathfrak{p}\overline{\mathfrak{p}}$ by (\cite{Lan2} Chapter 10. \S4 Theorem 10), where $\overline{\mathfrak{p}}$
is the complex conjugation of $\pi$.
Since $\psi(\mathfrak{p})$ is the value of $\psi$ at an id\`ele with 1's in its archimedean components, we have $\psi(\mathfrak{p})\in \mathcal{O}_{K}$.
Then we have $\pi=\psi(\mathfrak{p})$ and
$p=\pi \bar{\pi}$ ($\bar{\pi}$ is the complex conjugation of $\pi$). 
Now we take an immersion $\overline{K}\hookrightarrow \mathbb{C}_{p}$ satisfying $|\pi|<1$ in $\mathbb{C}_{p}$.\\
\quad For a natural number $N$,
we let $\mathfrak{g'}:=\mathfrak{g}p^N$ be an integral ideal of $\mathcal{O}_{K}$ which is divisible by the conductor $\mathfrak{f}$ of $\psi$,
where $\mathfrak{g}=(g)$ is the integral ideal of $\mathcal{O}_{K}$ prime to $p$.\\
\quad  Let $I(\mathfrak{g}')$ is a group of fractional ideals of $\mathcal{O}_{K}$ prime to $\mathfrak{g}'$.
 Let $\varphi: I(\mathfrak{g}')\to \overline{K}^{\times}$ be an algebraic Hecke character of infinite type $(m,n)\in \Z^{2}$
 whose conductor divides $\mathfrak{g}'$: In other words, $\varphi$ is the group homomorphism satisfying
 $$
 \varphi((\alpha)):=\chi(\alpha)\alpha^{m}\overline{\alpha}^{n}\quad \text{for any}\ \alpha\in \mathcal{O}_{K}\ \text{prime to}\ \mathfrak{g}' 
 $$
for some finite character $\chi: (\mathcal{O}_{K}/\mathfrak{g}')^{\times}\to \overline{K}^{\times}$.
The classical complex Hecke $L$-function of $\varphi$ is defined by
\begin{equation}
L_{\mathfrak{g}'}(s,\varphi):=\sum_{(\mathfrak{a}, \mathfrak{g}')=1}\frac{\varphi(\mathfrak{a})}{N(\mathfrak{a})^s}\qquad \text{for}\ \operatorname{Re}(s)\gg 0
\label{The classical complex Hecke L-function}
\end{equation}
where the sum runs over all integral ideals $\mathfrak{a}$ of $\mathcal{O}_{K}$ prime to $\mathfrak{g}'$.
By defining $\varphi(\mathfrak{a})$ by 0 if $\mathfrak{a}$ is not prime to $\mathfrak{g}'$, we can consider $\varphi$ as a function
on the group of all fractional ideals of  $\mathcal{O}_{K}$. If $\operatorname{Re}(s)>(m+n)/2+1$, the Hecke $L$-function (\ref{The classical complex Hecke L-function})
converges absolutely.
The analytic continuation and functional equation of $L_{\mathfrak{g}'}(s,\varphi)$ is well known.
If we put
\begin{equation*}
\widehat{L}_{\mathfrak{g}'}(s,\varphi):=\frac{(d_{K}N(\mathfrak{g}'))^{s/2}\Gamma(s-\min\{m,n\})L_{\mathfrak{g}'}(s,\varphi)}{(2\pi)^{s-\min\{m,n\}}}
\end{equation*}
for $\pi=3.1415\cdots$ and discriminant $d_{K}$ of $K$,
we have
\begin{equation}
\widehat{L}_{\mathfrak{g}'}(s,\varphi)=W\cdot \widehat{L}_{\mathfrak{g}'}(1+m+n-s,\overline{\varphi}),
\label{functional equation of L-function}
\end{equation}
where $W$ is a constant of absolute value 1, called the Artin root number.
Since $\Gamma$-functions of both sides of (\ref{functional equation of L-function}) have no poles on
$\{(m,n)\in \Z^{2}\ |\ m<0, n\geq 0 \}$ or $\{(m,n)\in \Z^{2}\ |\ n<0, m\geq 0 \}$, following Deligne (\cite{De} D\'{e}finition 1.3),
these two sets are \textbf{critical} domains.
In addition, for integers $m,n$ with $m< 0$ and $n\geq 0$ (resp. $n< 0$ and $m\geq 0$), by Damerell's theorem (\cite{BK} Corollary 2.12), we have
\begin{equation}
\frac{L_{\mathfrak{g}'}(0,\varphi)}{\Omega^{n-m}}\in \overline{\Q} \qquad 
\left(\text{resp.}\ \frac{L_{\mathfrak{g}'}(0,\overline{\varphi})}{\Omega^{m-n}}\in \overline{\Q} \right).
\label{algebraicity of L-function at critical}
\end{equation}
(\ref{algebraicity of L-function at critical}) asserts that Deligne's conjecture (\cite{De} Conjecture 1.8) holds with respect to $L$-functions associated to algebraic Hecke characters.
Main theorem of this article is the theorem that the $p$-adic analogue of $L_{\mathfrak{g}'}(0,\varphi)/\Omega^{n-m}$
can be expressed by using the $p$-adic Eisenstein-Kronecker series 
in the non-critical domain.
In order to achieve our purpose, we use the $p$-adic Eisenstein-Kronecker series as Coleman functions constructed by \cite{BFK},
which we call the \textbf{Coleman Eisenstein-Kronecker series}.\\
\quad By the conditions $(g,p)=1$ and $(\pi,\overline{\pi})=1$, by the Chinese remainder theorem, we have
$$
(\mathcal{O}_{K}/\mathfrak{g}')^{\times}\cong (\mathcal{O}_{K}/\mathfrak{g})^{\times}\times
(\mathcal{O}_{K}/(\pi^{N}))^{\times}\times(\mathcal{O}_{K}/(\overline{\pi}^N))^{\times}
$$
Therefore if we restrict the finite character $\chi$ on $(\mathcal{O}_{K}/\mathfrak{g}')^{\times}$ respectively
to $\chi_{\mathfrak{g}}: (\mathcal{O}_{K}/\mathfrak{g})^{\times} \to \overline{K}^{\times}$,
$\chi_{1}: (\mathcal{O}_{K}/(\pi^{N}))^{\times}\to \overline{K}^{\times}$,
and $\chi_{2}: (\mathcal{O}_{K}/(\overline{\pi}^{N}))^{\times}\to \overline{K}^{\times}$,
we can decompose $\chi$ by
$$
\chi(\alpha)=\chi_{\mathfrak{g}}(\alpha)\chi_{1}(\alpha)\chi_{2}(\alpha).
$$
We extend $\chi_{\mathfrak{g}}, \chi_{1}, \chi_{2}$ respectively into characters with $\C_{p}$-values by using
an inclusion map $i: \overline{K}^{\times}\hookrightarrow \C_{p}^{\times}$.\\
\quad Now we put $\X:=\varprojlim_{n}(\mathcal{O}_{K}/\mathfrak{g}p^n \mathcal{O}_{K})^{\times}$.
We define the \textbf{$p$-adic character} $\phi_{p}: \X\to \C_{p}^{\times}$ by (\cite{dS} Chapter II \S1 (5)).
Similarly to (\cite{dS} Chapter II \S4.16 (49)) or (\cite{BKT} \S2.4), for a measure $\mu_{g}$ on $\X$ (for details, see (\ref{definition of the measure on X})),
we define the value of the \textbf{$p$-adic $L$-function} at the $p$-adic character $\phi_{p}: \X\to \C_{p}^{\times}$ by
\begin{equation*}
L_{p}(\phi_{p}):=\int_{\X}\phi_{p}(\alpha)d\mu_{g}(\alpha).
\end{equation*}
Now for the $p$-adic character $\varphi_{p}: \X\to \C_{p}^{\times}$ defined by 
$\varphi_{p}(\alpha):=\varphi((\alpha))$
for any $\alpha\in \mathcal{O}_{K}$ prime to $\mathfrak{g}p$,
since 
$$
\X\cong (\mathcal{O}_{K}/\mathfrak{g})^{\times}\times (\mathcal{O}_{K}\otimes_{\Z}\Z_{p})^{\times},
$$
we have $\varphi_{p}=\chi_{\mathfrak{g}}\chi_{1}\chi_{2}\kappa_{1}^{m}\kappa_{2}^{n}$ as $p$-adic characters on $\X$,
where $\kappa_{1} , \kappa_{2}$ are the projections to the first and second factors of the following isomorphism:
\begin{equation}
(\mathcal{O}_{K}\otimes_{\Z}\Z_{p})^{\times}\cong \mathcal{O}_{K_{\mathfrak{p}}}^{\times}\times \mathcal{O}_{K_{\overline{\mathfrak{p}}}}^{\times}
\overset{\cong}{\longrightarrow}\Z_{p}^{\times} \times \Z_{p}^{\times}
 \quad \alpha\overset{\sim}{\longmapsto}(\kappa_{1}(\alpha),\kappa_{2}(\alpha)).
\label{the projections to the first and second factors}
\end{equation}
Note that $\chi_1$, $\chi_2$ which are components of $\varphi_{p}$ can be regarded as characters on $\X$ by the following liftings of the natural projections:
\begin{align*}
\X\joinrel\relbar\joinrel\relbar\joinrel\twoheadrightarrow \varprojlim_{n}(\mathcal{O}_{K}/p^n \mathcal{O}_{K})^{\times}\cong
(\mathcal{O}_{K}\otimes_{\Z}\Z_{p})^{\times}
\overset{\kappa_{1}}{\joinrel\relbar\joinrel\relbar\joinrel\twoheadrightarrow}
\mathcal{O}_{K_{\mathfrak{p}}}^{\times}
\joinrel\relbar\joinrel\relbar\joinrel\twoheadrightarrow (\mathcal{O}_{K}/(\pi^N))^{\times}
\overset{\chi_{1}}{\joinrel\relbar\joinrel\relbar\joinrel\longrightarrow}\overline{K}^{\times}
\overset{i}{\hookrightarrow} \C_{p}^{\times}\\
\X\joinrel\relbar\joinrel\relbar\joinrel\twoheadrightarrow \varprojlim_{n}(\mathcal{O}_{K}/p^n \mathcal{O}_{K})^{\times}\cong
(\mathcal{O}_{K}\otimes_{\Z}\Z_{p})^{\times}
\overset{\kappa_{2}}{\joinrel\relbar\joinrel\relbar\joinrel\twoheadrightarrow}
\mathcal{O}_{K_{\overline{\mathfrak{p}}}}^{\times}
\joinrel\relbar\joinrel\relbar\joinrel\twoheadrightarrow (\mathcal{O}_{K}/(\overline{\pi}^N))^{\times}
\overset{\chi_{2}}{\joinrel\relbar\joinrel\relbar\joinrel\longrightarrow}\overline{K}^{\times}
\overset{i}{\hookrightarrow} \C_{p}^{\times}.
\end{align*}
Let $\Omega_{\mathfrak{p}}\in \mathcal{O}_{\C_{p}}^{\times}$
 be the \textbf{$p$-adic period} obtained from the isomorphism between the formal group of the elliptic curve and the multiplicative formal group
(for details, see (\ref{isomorphism between formal groups})).
By calculation, we know that for integers $m,n$ with $m<0$ and $n\geq 0$.
$$
\frac{L_{p}(\varphi_{p})}{\Omega_{\mathfrak{p}}^{n-m}}=(\text{interpolation factor})\times \frac{L_{\mathfrak{g}'}(0,\varphi)}{\Omega^{n-m}}
$$
holds. Hence $L_{p}(\varphi_{p})/\Omega_{\mathfrak{p}}^{n-m}$ is the $p$-adic analogue of $L_{\mathfrak{g}'}(0,\varphi)/\Omega^{n-m}$.\\
\quad K. Bannai, S. Kobayashi, and T. Tsuji related the non-critical values of the
$p$-adic $L$-function and $p$-adic Eisenstein-Kronecker series constructed by the measure
when the conductor of algebraic Hecke character is \textit{not divisible} by $p$.
However when the conductor of the algebraic Hecke character is \textit{divisible} by $p$, 
we \textit{cannot express} special values of the $p$-adic $L$-function by using the
$p$-adic Eisenstein-Kronecker series constructed by the measure.
So we use the $p$-adic Eisenstein-Kronecker series constructed by the \textit{Coleman integration} established in \cite{BFK}.
For any integers $m,n$ with $n\geq 0$, let $E_{m,n}^{\operatorname{col}}(z)$ be the Coleman Eisenstein-Kronecker series on $E(\C_{p})\setminus [0]$.
We expressed non-critical values of the $p$-adic $L$-function by using the Coleman Eisenstein-Kronecker series as follows:
\begin{Thm}[=Theorem \ref{My Main Theorem}]
Let $m,n$ be any integers with $n\geq 0$.
For a primitive $\mathfrak{g}'$-torsion point $\xi_{\mathfrak{g}'}:=i_{*}(\Omega C/gp^N)$
that $C$ is a special constant of $z$ contained in the following equation, we have
\begin{align*}
\frac{L_{p}(\varphi_{p})}{\Omega_{\mathfrak{p}}^{n-m}}=
\frac{g^{-1} n! (-1)^{m+n+1}}{\tau(\overline{\chi_{1}})\overline{\pi}^N}
\sum_{z\in (\mathcal{O}_{K}/\mathfrak{g}')^{\times}}\chi(z)
E_{m+1, n+1}^{\operatorname{col}}(\xi_{\mathfrak{g}'}z),
\end{align*}
where $\tau(\overline{\chi_{1}})$ is the Gauss sum defined by Lemma \ref{Properties of Gauss sum}
and $i_{*}: E(\overline{K})\hookrightarrow E(\C_{p})$ is the inclusion map induced by the inclusion map $i: \overline{K}\hookrightarrow \C_{p}$.
Note that $\Omega C/gp^N$ is the element in $E(\overline{K})$ through the isomorphism 
$\mathfrak{g}'^{-1}\Gamma/\Gamma\cong E(\C)[\mathfrak{g}']\cong E(\overline{K})[\mathfrak{g}']$.
Here $E(\overline{K})[\mathfrak{g}']$ (resp. $E(\C)[\mathfrak{g}']$) is a subgroup of $E(\overline{K})$ (resp. $E(\C)$), which
consists of $\mathfrak{g}'$-torsion points.
\end{Thm}

\section{Coleman integration theory and its applications}
\subsection{A brief review of Coleman integration theory}\leavevmode\\
\quad In this chapter, we give a brief review of Coleman integration theory.
Coleman constructed $p$-adic integration theory by using rigid analysis which Tate introduced in \cite{Ta}
and defined $p$-adic polylogarithms in his studies of $p$-adic analogue of Bloch's results
which related special values of $L$-functions of algebraic varieties to Quillen's $K$-groups with dilogarithms (see \cite{Col1}, \cite{Blo}).
For Tate's rigid analysis, see (\cite{BGR}, \cite{FdP}, or \cite{Ber} \S0).
Coleman functions are, roughly speaking, generalization of rigid analytic functions and given as power series which converge on each unit open disk.
For details, see \cite{Bes0}.
In addition, Besser gave generalization
of Coleman integrations on a smooth and proper algebraic variety which has a good reduction in \cite{Bes1}.\\
\quad In Coleman integration theory, there are two important properties.
\begin{itemize}
\item[(A)] The \textbf{uniqueness principle} holds. Thus we can consider analytic continuation. (for example, see (\cite{CdS} Corollary 2.4.5))
\item[(B)] We can locally integrate any differential forms and it is \textbf{unique} up to a constant by Frobenius invariance.
\end{itemize}
\quad In Tate's rigid analysis, (A) holds (for example, see (\cite{Ber}) Proposition 0.1.13). But (B) does not hold
because for example if $t=0$ is removed for a local parameter $t$, we cannot integrate $dt/t$ in the affinoid algebra.
So we will extend into a bigger ring so that we can integrate any differential forms.\\
\quad Let $\mathcal{O}_{\mathbb{C}_{p}}$ be an integer ring of the completion $\mathbb{C}_{p}$ of algebraic closure of $\mathbb{Q}_{p}$.
Then a residue field of $\mathbb{C}_{p}$ is $\overline{\mathbb{F}}_{p}$.
Let $X$ be a proper, smooth, and connected scheme of locally finite type of relative dimension 1 defined over $\mathcal{O}_{\mathbb{C}_{p}}$
with a good reduction. Let $X(\mathbb{C}_{p})$ be a generic fiber and $X(\overline{\mathbb{F}}_{p})$ be a special fiber.
According to (\cite{Ber} Proposition 0.3.5), $X(\mathbb{C}_{p})$ is isomorphic to a rigid analytic $\mathbb{C}_{p}$-space $X^{\operatorname{an}}$
obtained by a rigid analyzation of $X$, so we may also denote $X(\mathbb{C}_{p})$ by $X(\mathbb{C}_{p})^{\operatorname{an}}$.\\
\quad Let $Y\subset X$ be an open affine subscheme defined over $\mathcal{O}_{\mathbb{C}_{p}}$ which is proper, smooth, connected and has a good reduction.
Let $Y(\overline{\mathbb{F}}_{p})$ be a special fiber of $Y$. 
Then we can take finite points $e_1,\cdots, e_n$ such that 
$X(\overline{\mathbb{F}}_{p})\setminus Y(\overline{\mathbb{F}}_{p})=\{ e_1,\cdots, e_n \}$.\\
\quad For a $\overline{\mathbb{F}}_{p}$-subscheme $S\subset X(\overline{\mathbb{F}}_{p})$, 
let $]S[:=\operatorname{sp}^{-1}(S)\subset X(\mathbb{C}_{p})^{\operatorname{an}}$
be a tube of $S$, where $\operatorname{sp}$ is the specialization map 
$$
\operatorname{sp}:\ X(\mathbb{C}_{p})^{\operatorname{an}}\overset{\text{reduction}}{\relbar\joinrel\relbar\joinrel\relbar\joinrel\relbar\joinrel\longrightarrow} X(\overline{\mathbb{F}}_{p}).
$$
In particular, for a closed point $x\in X(\overline{\mathbb{F}}_{p})$, $]x[:=\operatorname{sp}^{-1}(x)$
is a unit open disk by (\cite{Ber} Proposition 1.1.1). In other words, $]x[\ \cong \{z\in \mathbb{C}_{p}\ |\ |z|<1 \}$.
We denote a local parameter $z$ around $x$ by $z_x$. 
For $0<r\leq 1$, we put $U_{r}:=X(\C_{p})^{\operatorname{an}}\setminus \cup_{i=1}^{n} D(\widetilde{e}_{i},r)^{-}$.
Here, $\widetilde{e}_{i}\in X(\C_{p})$ is a lift of $e_{i}$ and $D(\widetilde{e}_{i},r)^{-}$ is a closed disk centered at $\widetilde{e}_{i}$ with radius $r$.\\
\quad Let $\mathcal{O}_{X^{\operatorname{an}}}$ be a structure sheaf of the rigid analytic $\mathbb{C}_{p}$-space
 $X^{\operatorname{an}}=X(\mathbb{C}_{p})^{\operatorname{an}}$. Let $U:=\operatorname{sp}^{-1}(X(\overline{\mathbb{F}}_{p}))=\varprojlim_{r\to 1}U_r$.
 Let $A(U)$ be a subset of a locally rigid analytic function $\mathscr{L}(U)$ on $U$ such that
 $$
 A(U):=\{f\in \mathscr{L}(U)\ |\  f|_{X^{\operatorname{an}}}\in  \mathcal{O}_{X^{\operatorname{an}}}(X^{\operatorname{an}})\}
 $$
where 
$\mathcal{O}_{X^{\operatorname{an}}}(X^{\operatorname{an}})=\Gamma(X^{\operatorname{an}}, \mathcal{O}_{X^{\operatorname{an}}})$
 is an algebra of rigid analytic functions.
Let $\Omega^{1}(U)$ be a space of 1-forms. 
These are respectively rings of overconvergent functions and overconvergent 1-forms in the sense of Monsky-Washnitzer. 
In other words, we can regard $A(U)$ as $\Gamma(]Y(\overline{\mathbb{F}}_{p})[, j^{\dagger}\mathcal{O}_{]Y(\overline{\mathbb{F}}_{p})[})$ and 
$\Omega^{1}(U)$ as $\Gamma(]Y(\overline{\mathbb{F}}_{p})[, j^{\dagger}\Omega^{1}_{]Y(\overline{\mathbb{F}}_{p})[})$,
where for an open immersion $j: S \hookrightarrow \overline{S}$ corresponding to a closed subscheme $\overline{S}\subset X(\overline{\mathbb{F}}_{p})$
and an open $S$ of $\overline{S}$, $j^{\dagger}$ is a functor defined by (\cite{Ber} \S2.1 (2.1.1.1)).\\
\quad A \textbf{branch} of $p$-adic logarithms is any locally analytic homomorphism
 $\log: \mathbb{C}_{p}^{\times}\to \mathbb{C}_{p}^{+}$ with the usual expansion for $\log$ around 1.
Such a function is determined by choosing $\pi\in \mathbb{C}_{p}$ such that $|\pi| < 1$ and declaring $\log(\pi) = 0$.
Coleman's $p$-adic integration theory depends on the choice of the branch of the $p$-adic logarithms.
We choose such a branch ``$\log$" of $p$-adic logarithms. We define
\begin{align*}
A_{\log}(]x[)&:=
\begin{cases}
A(]x[) & \text{if}\ x\in Y(\overline{\mathbb{F}}_{p})\\
\lim_{r\to 1}A(]x[\cap U_r)[\log(z_x)]
& \text{if}\ x\in X(\overline{\mathbb{F}}_{p})\setminus Y(\overline{\mathbb{F}}_{p}) 
\end{cases}&\\
\Omega_{\log}^{1}(]x[)&:=A_{\log}(]x[)dz_x&
\end{align*}
Here, note that 
if $x\in Y(\overline{\mathbb{F}}_{p})$, then $A(]x[)$ is the ring $\mathcal{O}_{]x[}(]x[)$ consisting of formal power series 
$f(z_x)=\sum_{n=0}^{\infty}a_{n}z_{x}^{n}$ which converges on $\{z_x\in \C_{p}\ |\ |z_x|<1  \}$,
and if $x\in X(\overline{\mathbb{F}}_{p})\setminus Y(\overline{\mathbb{F}}_{p})$, then 
formal power series $f(z_x)=\sum_{n=-\infty}^{\infty}a_{n}z_{x}^{n}$ which converges on $\{z_x\in \C_{p}\ |\ r<|z_x|<1 \}$ for some $r<1$.
We define rings of locally analytic functions and 1-forms on $U$ by 
$$
A_{\operatorname{loc}}(U):=\prod_{x\in X(\overline{\mathbb{F}}_{p})}A_{\log}(]x[),\qquad
\Omega^{1}_{\operatorname{loc}}(U):=\prod_{x\in X(\overline{\mathbb{F}}_{p})}\Omega^{1}_{\log}(]x[).
$$
These are independent of the choice of $z_x$. We can define a differential $d: A_{\operatorname{loc}}(U) \to \Omega^{1}_{\operatorname{loc}}(U)$
in the natural way. Then $d: A_{\operatorname{loc}}(U) \to \Omega^{1}_{\operatorname{loc}}(U)$ is surjective. 
The point is that we are able to integrate $dz/z$ by adding logarithms.
So we can integrate any elements in 
$\Omega^{1}_{\operatorname{loc}}(U)$ i.e. (B) holds. But 
since $\operatorname{Ker}(d)=\prod_{x\in X(\overline{\mathbb{F}}_{p})} \mathbb{C}_{p}$, we do not have the notion of analytic continuation yet i.e. (A)
does not hold. 
\begin{Def}[Coleman function]
Coleman defined a subalgebra $M(U)\subset A_{\operatorname{loc}}(U)$ equipped with an integration map
$$
\int:\ M(U)\otimes_{A(U)}\Omega^{1}(U)\to M(U)/\mathbb{C}_{p} \qquad \omega\mapsto F_{\omega}:=\int \omega
$$
which is one of $\mathbb{C}_{p}$-linear maps, in order to obtain the notion of analytic continuation, with  the surjectivity of $d$ keeping as follows.
A map $\int$ is characterized by three properties:
\begin{itemize}
\item[i)] (The existence of a primitive function)\quad $dF_{\omega}=\omega$
\item[ii)] (Frobenius invariance)\quad For a Frobenius automorphism  $\phi:\ U\to U$, we have
$$
\int (\phi^{*} (\omega) )=\phi^{*}\left(\int \omega \right)
$$
\item[iii)] $\int dg=g+\mathbb{C}_{p}$\quad for $g\in A(U)$.
\end{itemize}
We call $M(U)$ a space of \textbf{Coleman functions} on $U$. As for the construction of  such a space $M(U)$, see (\cite{Bes0} \S2).
\end{Def}
In summary, when $f$ is a function in $A_{\operatorname{loc}}(U)$ and $P(x)$ is a polynomial with $\C_{p}$-coefficients whose roots do not contain the roots of 1,
if $df \in M(U)\otimes_{A(U)}\Omega^{1}_{\operatorname{loc}}(U)$ and $P(\phi^{*})f \in M(U)$, then we have $f\in M(U)$.
Note that we extend the classes of integrable differential forms so that the integration is unique up to a constant in $\mathbb{C}_{p}$, not in 
$\prod_{x\in X(\overline{\mathbb{F}}_{p})}\mathbb{C}_{p}$. In other words, we have an exact sequence
$$
0\longrightarrow \C_{p}\longrightarrow M(U)\overset{d}{\longrightarrow} M(U)\otimes_{A(U)}\Omega^{1}(U)\longrightarrow 0.
$$
\quad The entire theory turns out to be independent of the choice of $\phi$.
The important idea is to extend the classes of the integrable differential forms from $d(A(U))$ by using Frobenius invariance.

\subsection{Applications of Coleman integrations}\leavevmode\\
Put $U:=\mathbb{P}^{1}(\mathbb{C}_{p})\setminus \{ 0,1,\infty \}$. Coleman defined $p$-adic polylogarithms recursively as follows:
\begin{Def}[$p$-adic polylogarithm]
Let $k$ be an integer. 
If $k\geq 0$,
we define a locally analytic function 
$$
\ell_{k}\in M(U)\quad (k\geq 0 )
$$
satisfying
\begin{itemize}
\item [i)] $\displaystyle{\ell_{0}(z)=\frac{z}{1-z}}$
\item [ii)] $\displaystyle{d\ell_{k}(z)=\ell_{k-1}(z)\frac{dz}{z}}$
\item [iii)] $\displaystyle{\lim_{z\to 0}\ell_{k}(z)=0}$.
\end{itemize}
$\ell_{k}(z)$ is an analytic function $\ell_{k}(z)=\sum_{n=1}^{\infty}\frac{z^n}{n^k}$ on $|z|<1$.
The existence and uniqueness of $\ell_{k}$ is insured by (\cite{Bre} Corollaire 2.2.2.1).
If $k\leq 0$, we take $\ell_{k}\in A(U)$ satisfying i), ii), iii). 
This $\ell_{k}$ is the \textbf{$p$-adic polylogarithm} defined by Coleman \cite{Col1}.
\label{existence and uniqueness of p-adic polylog}
\end{Def}
The important application of the $p$-adic polylogarithm is that Kubota-Leopoldt $p$-adic $L$-function $L_{p}(s,\chi)$ associated to
 a non-trivial Dirichlet character $\chi$ can be written as the sum of the $p$-adic polylogarithms at positive integers (i.e. non-critical values).
 It is well known that Kubota-Leopoldt $p$-adic $L$-function $L_{p}(s,\chi)$ is obtained by interpolating at negative integers (i.e. critical points) of 
 the complex Dirichlet $L$-function (see \cite{Iwa} \S3 Theorem 3 ii)).
\begin{Thm}[Coleman \cite{Col1} \S7]
Let $p$ be an odd prime. Let $\chi:\ (\mathbb{Z}/d\mathbb{Z})^{\times}\to \mathbb{C}_{p}^{\times} $ be a Dirichlet character
with a conducter $d>1$ such that $d$ is prime to $p$, and $\omega: \mathbb{Z}_{p}^{\times}\to (\mathbb{Z}/p\mathbb{Z})^{\times}$ be a Teichm\"{u}ller character.
For all integers $k\geq 1$, we have
$$
L_{p}(k,\chi\omega^{1-k})=\left(1-\frac{\chi(p)}{p^k}\right)\frac{g(\chi,\zeta)}{d}\sum_{a=1}^{d-1}\overline{\chi(a)}\ell_{k}(\zeta^{-a}),
$$  
where $\zeta$ is a primitive $d$-th root of 1 and $g(\chi,\zeta)$ is the Gauss sum defined by $g(\chi,\zeta)=\sum_{a=0}^{d-1}\chi(a)\zeta^{a}$.
\label{p-adic L-function and p-adic polylog}
\end{Thm}
Substituted for $k=1$, 
Theorem \ref{p-adic L-function and p-adic polylog} is reduced to (\cite{Iwa} \S 5 Theorem 3).
Note that Theorem \ref{p-adic L-function and p-adic polylog} is the $p$-adic analogue of the classical formula
$$
L(k,\chi)=\frac{g(\chi,\zeta)}{d}\sum_{a=1}^{d-1}\overline{\chi(a)}\ell_{k}(\zeta^{-a}).
$$
We can show this formula by using two properties of the Gauss sum
$$
\sum_{a=0}^{d-1}\overline{\chi(a)}\zeta^{-an}=\chi(n)\overline{g(\chi,\zeta)}
$$
and 
$$
g(\chi,\zeta)\overline{g(\chi,\zeta)}=d.
$$
Main theorem of this article is the elliptic analogue of Theorem \ref{p-adic L-function and p-adic polylog}.

\section{Review of the classical Eisenstein-Kronecker series and its $p$-adic analogue}
In this section, we review the definition of the classical Eisenstein-Kronecker series by A. Weil \cite{We} and of the 
$p$-adic Eisenstein-Kronecker series as the Coleman function by K. Bannai, H. Furusho, and S. Kobayashi in \cite{BFK}. 
The classical Eisenstein-Kronecker series is the elliptic analogue of 
the classical complex polylogarithm by using the Bloch-Wigner-Ramakrishnan polylogarithm which is invariant under 
the map $z\mapsto qz$ on an elliptic curve $\C^{\times}/q^{\Z}$ where $q=e^{2\pi i\tau}$ for $\operatorname{Im}(\tau)>0$ 
(see \cite{Zag} Theorem 1).
The $p$-adic Eisenstein-Kronecker series as the Coleman function i.e. Coleman Eisenstein-Kronecker series
is defined by constructing with generating functions appeared in Laurent coefficients of the Kronecker theta function. \\
\subsection{Review of the classical Eisenstein-Kronecker series}\leavevmode\\
\quad Recall that the definition of the classical Eisenstein-Kronecker series (\cite{We} VIII \S12).\\
Let $\Gamma\subset \mathbb{C}$ be a lattice, $\varpi=3.1415\cdots$ be the circular constant,
$A(\Gamma)=(\text{Area of}\ \mathbb{C}/\Gamma)/\varpi$,
$\chi_{w}(z)_{\Gamma}:=\exp((z\overline{w}-w\overline{z})/A(\Gamma))$ for any $z,w\in \mathbb{C}$.\\
\quad In particular, 
if we can write $\Gamma:=\mathbb{Z}\omega_{1}\oplus \mathbb{Z}\omega_{2}\subset \mathbb{C}$ with $ \operatorname{Im}(\omega_{2}/\omega_{1})>0$,
 note that
$$
A(\Gamma)=\frac{1}{\varpi}\operatorname{Im}(\omega_{2}/\omega_{1})=\frac{1}{2\varpi i}(\omega_{2}\overline{\omega}_{1}-\omega_{1}\overline{\omega}_{2}).
$$
In addition, by direct calculations, we know the following properties.
\begin{itemize}
\item [i)] $\chi_{w}(z)_{\Gamma}=\chi_{z}(-w)_{\Gamma}=\chi_{z}(w)_{\Gamma}^{-1}$
\item [ii)]$\chi_{w}(az)_{\Gamma}=\chi_{\overline{a}w}(z)_{\Gamma}$ for any $a\in\mathbb{C}$
\item [iii)] $z\in \Gamma \iff  \chi_{\gamma}(z)_{\Gamma}=1$ for any $\gamma\in \Gamma\ $.
\end{itemize}

\begin{Def}[Eisenstein-Kronecker-Lerch series]\leavevmode\\
Let $a$ be an integer and $z_{0},w_{0}\in \mathbb{C}$ be complex numbers. The \textbf{Eisenstein-Kronecker-Lerch series} is defined by 
$$
K_{a}^{*}(z_{0}, w_{0}, s; \Gamma):=\sum_{\gamma\in\Gamma\setminus \{ -z_{0}\}}\frac{(\overline{z}_{0}+\overline{\gamma})^{a}}{|z_{0}+\gamma|^{2s}}\chi_{w_{0}}(\gamma)_{\Gamma} \quad (s\in \mathbb{C}).
$$
This series converges absolutely for $\operatorname{Re}(s)>a/2+1$.
\label{Eisenstein-Kronecker-Lerch series}
\end{Def}
Hereafter, by abuse of notations, we omit ``$\Gamma$" except the case where we want to express the lattice clearly.
$K_{a}^{*}(z_{0}, w_{0}, s)$ has the following important properties.

\begin{Prop}
Let $a$ be an integer and $z_{0},w_{0}\in \mathbb{C}$ be complex numbers.
\begin{itemize}
\item[i)] $K_{a}^{*}(z_{0}, w_{0}, s)$ can be continued meromorphically on $\mathbb{C}$ as a function of $s$.
Moreover, if $a=0$ and $w_{0}\in \Gamma$, $K_{a}^{*}(z_{0}, w_{0}, s)$ has a simple pole at $s=1$.
\item[ii)] $K_{a}^{*}(z_{0}, w_{0}, s)$ has a functional equation:
$$
\Gamma(s)K_{a}^{*}(z_{0}, w_{0}, s)=A^{a+1-2s}\Gamma(a+1-s)K_{a}^{*}(w_{0}, z_{0}, a+1-s)\chi_{z_{0}}(w_{0}),
$$
where $\displaystyle{\Gamma(s)=\int_{0}^{\infty}e^{-t}t^{s-1}dt\quad (\operatorname{Re}(s)>0)}$ is a Gamma function.
\end{itemize}
\label{functional equation of Eisenstein-Kronecker-Lerch series}
\end{Prop}
\begin{proof}
If $a\geq 0$, see (\cite{We} VIII \S13). If $a\leq 0$, see (\cite{BKT} Proposition 2.4.).
\end{proof}

\begin{Def}[Eisenstein-Kronecker number]\leavevmode\\
Let $z_0,w_0\in \mathbb{C}$ and we take $a,b\in \mathbb{Z}$ as $(a,b)\neq (1,-1)$ if $w_0\in \Gamma$.
The \textbf{Eisenstein-Kronecker numbers} $e_{a,b}^{*}(z_{0},w_{0})$ is defined by
$$
e_{a,b}^{*}(z_{0},w_{0}):=K_{a+b}^{*}(z_{0},w_{0},b)=
\sum_{\gamma\in \Gamma\setminus \{ -z_{0}\}}\frac{(\overline{z}_{0}+\overline{\gamma})^{a}}{(z_{0}+\gamma)^{b}}\chi_{w_{0}}(\gamma).
$$
For $(a,b)=(0,0)$, we have $e_{0,0}^{*}(z_{0},w_{0}):=K_{0}^{*}(z_{0},w_{0},0)=-\chi_{z_{0}}(w_{0})$.
\end{Def}  
We define the Kronecker theta function. The Kronecker theta function was defined by using the reduced theta function associated to the divisor $[0]$
(i.e. the holomorphic pseudo-periodic function with the Appell-Humbert data)
by using that a group of isomorphism classes of invertible sheaves on the torus $\C/\Gamma$ is classified by Appell-Humbert's theorem.
For details, see (\cite{BK} Example 1.9).
\begin{Def}[Kronecker theta function]
Let $\theta(z)$ be a reduced theta function associated to the divisor $[0]$ defined by (\cite{BK} Example 1.9).
$\theta(z)$ is characterized by $\theta'(0)=1$.\\
\quad By using this $\theta(z)$, the Kronecker theta function is defined as follows.
For any $z,w\in \mathbb{C}$, we define the \textbf{Kronecker theta function} $\Theta(z,w)$ by
$$
\Theta(z,w):=\frac{\theta(z+w)}{\theta(z)\theta(w)}.
$$
\quad In addition, for $z_{0},w_{0}\in \mathbb{C}$, we define
\begin{equation}
\Theta_{z_{0},w_{0}}(z,w):=\exp\left( -\frac{z_{0}\overline{w}_{0}}{A} \right)\exp\left( -\frac{z\overline{w}_{0}+w\overline{z}_{0}}{A} \right)\Theta(z+z_{0}, w+w_{0}).
\label{twisted Kronecker theta function}
\end{equation}
\end{Def}
According to (\cite{BK} Proposition 1.16), $\Theta_{z_{0},w_{0}}(z,w)$ has the following distribution relation: 
Let $c,c'\in \Gamma$, $n$ be a natural number, and $z_0, w_0\in \C$. Then we have
\begin{align}
\label{distribution relation of Kronecker theta function 1}
&\sum_{w_n\in \pi^{-n}\Gamma/\Gamma}\chi_{w_n}(c) \Theta_{z_0, w_0+w_n}(z,w)=
\pi^{n} \chi_{c}(w_0) \Theta_{(z_0-c)/\overline{\pi}^n, \pi^n w_0}(z/\overline{\pi}^n, \pi^n w)&\\
\label{distribution relation of Kronecker theta function 2}
&\sum_{z_m\in \pi^{-m}\Gamma/\Gamma}\chi_{c'}(z_m) \Theta_{z_0+z_m, w_0}(z,w)=
\pi^m  \Theta_{\pi^m z_0, c'/\overline{\pi}^m}(\pi^m z, w/\overline{\pi}^m)&\\
\notag
\end{align}
\quad $\Theta_{z_{0},w_{0}}(z,w)$ can be expanded to Laurent series of $z,w$ as the generating function of the following Eisenstein-Kronecker number.
\begin{Thm}
$\Theta_{z_{0},w_{0}}(z,w)$ has a Laurent expansion in the neighborhood of $(z,w)=(0,0)$, that is,
\begin{align}
\Theta_{z_{0},w_{0}}(z,w)=\chi_{z_{0}}(w_{0})\frac{\delta_{z_{0}}}{z}+\frac{\delta_{w_{0}}}{w}+
\sum_{a,b\geq 0}(-1)^{a+b}\frac{e_{a,b+1}^{*}(z_{0}, w_{0})}{a!A^a}z^{b}w^{a},
\label{Laurent expansion of Kronecker theta function}
\end{align}
where $\delta_{x}$ is defined by
$$
\delta_{x}=
\begin{cases}
1&(x\in \Gamma)\\
0&(\text{otherwise})
\end{cases}
$$
\label{Theorem of Laurent expansion of Kronecker theta function}
\end{Thm}
\begin{proof}
See (\cite{BK} \S1.14 Theorem 1.17).
\end{proof}
Substituting $w_{0}=0$ for the formula (\ref{Laurent expansion of Kronecker theta function}), 
we define a function $F_{z_{0},b}(z)$ as follows.
\begin{Def}
For any $z_{0}\in \mathbb{C}$, we define $F_{z_{0},b}$ by a function satisfying
\begin{equation}
\Theta_{z_{0},0}(z,w)=\sum_{b\geq 0}F_{z_{0},b}(z)w^{b-1}.
\label{definition of $F_{z_{0},b}$}
\end{equation}
If $z_{0}=0$, we define $F_{b}(z):=F_{0,b}(z)$. We observe $F_{b}(z)$ if $b=0,1$.\\
\quad $F_{0}(z)=1$.
Noting that $\Theta_{0,0}(z,w)=\Theta(z,w):=\theta(z+w)/\theta(z)\theta(w)$ and observing coefficients of $w^{0}$ in the formula 
(\ref{Laurent expansion of Kronecker theta function}) and (\ref{definition of $F_{z_{0},b}$}),
we find that $F_{1}(z)$ satisfies
 $$
F_{1}(z)=\lim_{w\to 0}(\Theta(z,w)-w^{-1})=\frac{\theta'(z)}{\theta(z)}.
$$ 
$F_{z_{0},b}(z)$ is dependent only on a choice of $z_{0}$ modulo $\Gamma$ because we have
$$
\Theta_{z_{0}+\gamma, 0}(z,w)=\exp\left[-\frac{w(\bar{z}_{0}+\bar{\gamma})}{A}\right]\Theta(z+z_{0}+\gamma,w)=\Theta_{z_{0},0}(z,w).
$$
\quad As we define later, the $p$-adic analogue of $F_{z_{0},b}$ for a variable $z_{0}$ is constructed as the Coleman function by glueing together each unit open disk.
By using the $p$-adic analogue of $F_{z_{0},b}$, we construct the $p$-adic analogue of the Eisenstein-Kronecker series $E_{m,n}$.
We have the Laurent expansion of $F_{z_{0},b}(z)$ from Theorem \ref{Theorem of Laurent expansion of Kronecker theta function}.
\begin{Cor}[Generating function]
For any $b\geq 0$, the Laurent series of $F_{z_{0},b}(z)$ at $z=0$ can be written as
$$
F_{z_{0},b}(z)=\frac{\delta_{z_{0},b}}{z}+\sum_{a\geq 0}(-1)^{a+b-1}\frac{e_{a,b}^{*}(0,z_{0})}{a!A^a}z^a,
$$
where  
$$
\delta_{x,b}=
\begin{cases}
1&(b=0\ \text{and}\  x\in \Gamma)\\
0&(\text{otherwise})
\end{cases}
$$
\end{Cor}

\begin{proof}
See (\cite{BKT} Corollary 2.11).
\end{proof}
\end{Def}

When we define the $p$-adic Eisenstein-Kronecker series later, we use the connection function $L_{n}(z)$ of $F_{b}(z)$ for $b\geq 0$.
We define the connection function $L_{n}(z)$ by 
$$
\Xi(z,w):=\exp(-F_{1}(z)w)\Theta(z,w)=\sum_{n\geq 0}L_{n}(z)w^{n-1}.
$$
\quad Since $F_{b}(z):=F_{0,b}(z)$,
$
\displaystyle{\Theta(z,w)=\Theta_{0,0}(z,w)=\sum_{b\geq 0}F_{b}(z)w^{b-1}},
$
and by giving ``$\exp$" Taylor expansion, we have
\begin{equation}
L_{n}(z)=\sum_{\substack{a+b=n\\ a\geq 0, b\geq 0}}\frac{(-F_{1}(z))^a F_{b}(z)}{a!}=\sum_{b=0}^{n}\frac{(-F_{1}(z))^{n-b}}{(n-b)!}F_{b}(z).
\label{expansion of $L_{n,z_0}$ by $F_{b}(z)$}
\end{equation}
By the translation by $\gamma\in\Gamma$ of $\Xi(z,w)$, $L_{n}(z)$ is the periodic function on $\mathbb{C}/ \Gamma$, i.e. the elliptic function and
the holomorphic function on $\mathbb{C}\setminus \Gamma$.
If $n=0,1$, by the formula (\ref{expansion of $L_{n,z_0}$ by $F_{b}(z)$}), we have
$$
L_{0}(z)=F_{0}(z)=1,\quad  L_{1}(z)=-F_{1}(z)F_{0}(z)+F_{1}(z)=0.
$$
\quad Since $\Theta_{z_{0},0}(z,w)=\exp(F_{z_{0},1}(z)w)\Xi(z+z_{0}, w)$, we have a relation
\begin{equation}
F_{z_{0},b}(z)=\sum_{n=0}^{b}\frac{F_{z_{0},1}(z)^{b-n}}{(b-n)!}L_{n}(z+z_{0})
\label{relation between $F_{z_{0},b}(z)$ and $L_{n}(z)$}
\end{equation}
between $F_{z_{0},b}(z)$ and $L_{n}(z)$.\\
\quad Now we assume that a complex torus has an algebraic model.
Let $K$ be an imaginary quadratic field and we fix an immersion $K\hookrightarrow \mathbb{C}$.
We define an elliptic curve $E$ over $K$ by a Weierstrass equation
\begin{equation*}
y^2=4x^3-g_{2}x-g_{3},\quad g_{2},g_{3}\in K,
\end{equation*}
and its invariant $\omega$ by $\omega=dx/y$.
By the uniformization theorem, there exists a period lattice $\Gamma\subset \mathbb{C}$ of $E$ satisfying an isomorphism
$$
\xi:\mathbb{C}/\Gamma \overset{\sim}{\longrightarrow}E(\mathbb{C}),\quad z\mapsto (\wp(z),\wp'(z)),
$$
where $\wp(z)$ is the Weierstrass $\wp$-function.
Then we have $\omega=d\wp(z)/\wp'(z)=dz$.
According to (\cite{BKT} Proposition 1.12), the connection function $L_{n}(z)$ is \textit{algebraic} since we have $L_{n}(z)\in K[\wp(z), \wp'(z)]$
 as the rational function on $E$ over $K$.
In addition, we assume that the elliptic curve \textit{$E$ has complex multiplication by the integer ring $\mathcal{O}_{K}$} of $K$.
Then by Damerell's theorem, $\Theta_{z_0, w_0}(z,w)$ and $F_{z_0, b}(z)$ are \textit{algebraic} in the following sense:
\begin{itemize}
\item
If $z_0, w_0$ correspond to torsion points in
$\C/\Gamma\cong E(\C)$, then 
\begin{equation}
\Theta_{z_0, w_0}(z,w)-\chi_{z_0}(w_0)\frac{\delta_{z_0}}{z}-\frac{\delta_{w_0}}{w}\in \overline{\Q}[[z,w]],
\label{Algebraicity of Kronecker theta function}
\end{equation}
where if $x\in \Gamma$ then $\delta_{x}=1$, and otherwise $\delta_{x}=0$ (see \cite{BFK} Theorem 2.13).
\item
If $z_{0}\in \mathbb{C}$ corresponds to a torsion point in $\mathbb{C}/\Gamma\cong E(\mathbb{C})$,
then
\begin{equation}
F_{z_{0},b}(z)-\frac{\delta_{z_{0},b}}{z}\in \overline{\mathbb{Q}}[[z]],
\label{Algebraicity of generating function}
\end{equation}
where if $b=1$ and $x\in \Gamma$ then $\delta_{x,b}=1$, and otherwise $\delta_{x,b}=0$ (see \cite{BFK} Corollary 2.14).
\end{itemize}
These algebraicities allow us to view this value as an element in $\C_{p}$ through the immersion $\overline{\Q}\hookrightarrow \C_{p}$.
In addition, since $E$ has complex multiplication by the integer ring $\mathcal{O}_{K}$, by (\cite{Sil2} II \S 1 Proposition 1.1),
there is a unique isomorphism
$$
[\cdot]: \mathcal{O}_{K}\overset{\sim}{\longrightarrow}\operatorname{End}(E)
$$
such that for any invariant differential $\omega=dx/y$ on $E$,  
 $$
 [\alpha]^{*}(\omega)=\alpha\omega \qquad \text{for all}\ \alpha\in \mathcal{O}_{K}.
 $$
 For any $0\neq \alpha\in \mathcal{O}_{K}$, let $E[\alpha]$ be a subgroup of $E(\overline{\mathbb{Q}})$ such that
$
E[\alpha]:=\{P\in E(\overline{\mathbb{Q}})|\ [\alpha]P=0  \}.
$
According to (\cite{BFK} Proposition 2.15), the function $F_{z_0, b}(z)$ is known to satisfy the following distribution relation with respect to $E[\alpha]$:
\begin{equation}
\sum_{z_{\alpha}\in E[\alpha]}F_{z_0+z_\alpha, b}(z)=\alpha\overline{\alpha}^{1-b}F_{\alpha z_0, b}(\alpha z)\quad  \text{for any}\ 0\neq\alpha \in \mathcal{O}_{K}.
\end{equation}

\subsection{Review of the $p$-adic analogue of the Eisenstein-Kronecker series}\leavevmode\\
\quad For an integer $b\geq 0$, we review that the $p$-adic analogue of $F_{z_0,b}$ is constructed as the Coleman function on an elliptic curve
 along \cite{BFK}.\\
\quad Let $E$ be an elliptic curve in \S0.
Let $t:=-2x/y$ be a formal parameter of $E$ at the origin.
Let $\widehat{E}$ be a formal group of $E$ for $t$ equipped with a maximal ideal of  a complete local ring of $\mathcal{O}_{K}$
and group operations $\oplus$. Let $\lambda: \widehat{E}\overset{\sim}{\to}\widehat{\G}_{a}$ be a normalized formal logarithm
 for an additive formal group $\widehat{\G}_{a}$. For a torsion point $z_0\in E(\overline{\Q})_{\operatorname{tors}}$, we define 
 $\widehat{F}_{z_0, b}(t)$ by
 $$
 \widehat{F}_{z_0, b}(t):=F_{z_0, b}(z)|_{z=\lambda(t)}=F_{z_0, b}(\lambda(t))
 $$
By formula (\ref{relation between $F_{z_{0},b}(z)$ and $L_{n}(z)$}) and $\lambda(z_0)=0$, we have
\begin{equation}
\widehat{F}_{z_0, b}(t)=\sum_{n=0}^{b}\frac{\widehat{F}_{z_0,1}(t)^{b-n}}{(b-n)!}\widehat{L}_{z_0, n}(t)
\end{equation}
where $\widehat{L}_{z_0, n}(t):=L_{n}(z+z_0)|_{z=\lambda(t)}$.\\
\quad By formula (\ref{Algebraicity of generating function}), we consider $\widehat{F}_{z_{0},b}(t)$ as 
power series with $\mathbb{C}_{p}$-coefficients through the immersion $\overline{\Q}\hookrightarrow \C_{p}$: In other words,
if $z_{0}\in E(\overline{\mathbb{Q}})$ is a torsion point with an order prime to $\mathfrak{p}$, according to (\cite{BKT} Proposition 2.16),
the following series converges
\begin{equation}
\widehat{F}_{z_{0},b}(t)-\frac{\delta_{z_{0},b}}{t}=\sum_{a\geq 0}(-1)^{a+b-1}\frac{e_{a,b}^{*}(0, z_{0})}{a!A^a}z^a\biggm|_{z=\lambda(t)}\in \mathbb{C}_{p}[[t]]
\label{BKT Proposition 2.16}
\end{equation}
converges on $B(0,1):=\{t\in \mathbb{C}_{p}|\ |t|<1 \}$ if $b\neq 1$ or $z_{0}\neq 0$.
In particular, this series is a rigid analytic function on $B(0,1)$.
Moreover $\widehat{F}_{1}(t):=\widehat{F}_{0,1}(t)$ converges on $\{t\in \mathbb{C}_{p}|\ 0<|t|<1 \}$.\\
In addition, we have a formula for translation by $\pi^{n}$-torsion points.

\begin{Prop}[Translation]
Let $z_{0}\in E(\overline{\mathbb{Q}})$ be a torsion point with an order prime to $\mathfrak{p}$. Then we have
 $$
\widehat{F}_{z_{0},b}(t\oplus t_{n})=\widehat{F}_{z_{0}+z_{n},b}(t),
$$
where $t_{n}\in \widehat{E}[\pi^n]$ is a $\pi^n$-torsion point and $z_{n}\in E(\overline{\mathbb{Q}})_{\operatorname{tors}}$ is the image of $t_{n}$ through
an immersion map 
$\widehat{E}(\mathfrak{m}_{\mathbb{C}_{p}})_{\operatorname{tors}}\hookrightarrow 
E(\overline{\mathbb{Q}})_{\operatorname{tors}}\hookrightarrow \mathbb{C}/\Gamma$.
Here $\mathfrak{m}_{\mathbb{C}_{p}}$ is a maximal ideal of an integer ring $\mathcal{O}_{\mathbb{C}_{p}}$ of  $\mathbb{C}_{p}$.
\label{translation of generating function}
\begin{proof}
See (\cite{BFK} Lemma 2.17).
\end{proof}
\end{Prop}
\quad Let $F\subset \mathbb{C}_{p}$ be a finite extension field of $K_{\mathfrak{p}}$.
By abuse of notations, we denote the extension of $E$ into an integer ring $\mathcal{O}_{F}$ of $F$ by again $E$. 
For
$\pi:=\psi(\mathfrak{p})$, 
let
$$
\phi: E\to E
$$
be a Frobenius automorphism  induced by a multiplication by $[\pi]$.\\
\quad Let $E(\mathbb{C}_{p}):=E(\mathbb{C}_{p})^{\operatorname{an}}$ be 
an extension into $\mathbb{C}_{p}$ of a rigid analytic $F$-space $E(F)^{\operatorname{an}}$,
and we fix a variable $z$ on $E(\mathbb{C}_{p})$.\\
\quad Each residue disk of $E(\mathbb{C}_{p})$ contains a Teichm\"uller representative,
where a Teichm\"uller representative is a unique element in the residue disk fixed by some power of proper Frobenius.
By choice of Frobenius morphism $\phi$, a Teichm\"uller representative is a torsion point $z_{0}$ with an order prime to $\mathfrak{p}$.\\
\quad For $t=-2x/y$, a unit open disk 
$
\{t\in \mathbb{C}_{p}|\ |t|<1  \}
$
expresses the residue disk $]0[:=\operatorname{sp}^{-1}(0)\subset E(\mathbb{C}_{p})$ containing a unit element with respect to 
group operations of  the elliptic curve $E$,
where 
$\operatorname{sp}:\ E(\mathbb{C}_{p})^{\operatorname{an}}\overset{\text{reduction}}{\relbar\joinrel\relbar\joinrel\relbar\joinrel\relbar\joinrel\longrightarrow} 
E(\overline{\mathbb{F}}_{p})$
 is a specialization map.\\
\quad Let $z_{0}\in E(\mathbb{C}_{p})$ be a torsion point with an order prime to $\mathfrak{p}$ and 
$\tau_{z_{0}}: E\to E$ be $\tau_{z_{0}}(z):=z+z_{0}$.
For this $\tau_{z_{0}}$, we define $]z_{0}[$ by 
\begin{equation}
]z_{0}[:=\tau_{z_{0}}(]0[).
\label{translation of rigid point}
\end{equation}
Then $]z_{0}[$ is a residue disk containing $z_{0}$.\\
\quad Let $U:=E(\mathbb{C}_{p})\setminus [0]$, where $[0]$ is a unit element in group laws of the elliptic curve.
If $t_{z_0}$ is a local parameter of $E$ at a point $z_0$, by the formula (\ref{BKT Proposition 2.16}),
$\widehat{F}_{z_{0},b}(t)$ defines an element in $A(]z_{0}[)$ via $]z_{0}[\cong \{t_{z_0}\in \mathbb{C}_{p}|\ |t_{z_0}|<1  \}$.
\begin{Lem}
We define $F_{1}^{\operatorname{col}}\in A_{\operatorname{loc}}(U)$ by 
$$
F_{1}^{\operatorname{col}}(z)|_{]z_{0}[}:=\widehat{F}_{z_{0},1}(t)\in A(]z_{0}[)\subset A_{\log}(]z_{0}[)
$$
on each residue disk $]z_{0}[$, where $z_{0}\in E(\overline{\mathbb{Q}})$ is a torsion point with an order prime to $\mathfrak{p}$.
Then $F_{1}^{\operatorname{col}}$ is a Coleman function on $U$.
\end{Lem}

\begin{proof}
See (\cite{BFK} Lemma 3.4.)
\end{proof}

The formula (\ref{expansion of $L_{n,z_0}$ by $F_{b}(z)$}) indicates that
$L_{n}$ is a rational function on $E$ with poles only at $[0]$ in $E$, hence
is in particular a Coleman function on $U$. The set of Coleman functions is
a ring, and we define $F_{b}^{\operatorname{col}}$ as follows.\\
\quad For $b\geq 1$, we define $F_{b}^{\operatorname{col}}$ by
if $b=1$ $F_{1}^{\operatorname{col}}(z)|_{]z_0[}:=\widehat{F}_{z_0,1}(t)$, and if $b>1$ 
$$
F_{b}^{\operatorname{col}}:=\sum_{n=0}^{b}\frac{(F_{1}^{\operatorname{col}})^{b-n}}{(b-n)!}L_{n}.
$$

According to (\cite{BFK} Proposition 3.6), $F_{b}^{\operatorname{col}}$ interpolates $F_{z_0,b}(z)$ on each unit open disk as follows:
For a torsion point $z_0\in E(\overline{\Q})_{\operatorname{tors}}$ with an order prime to $\mathfrak{p}$, we have
$$
F_{b}^{\operatorname{col}}(z)|_{]z_0[}=\widehat{F}_{z_0, b}(t)\in A_{\log}(]z_0[).
$$
In addition, $F_{b}^{\operatorname{col}}$ has the following distribution relation by (\cite{BFK} Proposition 3.7): 
\begin{equation}
\sum_{z_\alpha \in E[\alpha]}F_{b}^{\operatorname{col}}(z+z_\alpha)=\alpha\overline{\alpha}^{1-b}F_{b}^{\operatorname{col}}(\alpha z)
\qquad \text{for all}\ 0\neq \alpha \in \mathcal{O}_{K}.
\label{distribution relation of Coleman function of generating function}
\end{equation}
\quad In (\cite{BFK} \S3.3), K. Bannai, H. Furusho, and S. Kobayashi constructed the $p$-adic Eisenstein-Kroneker series as the Coleman function by using $F_{b}^{\operatorname{col}}$ whose constant term is chosen to satisfy the distribution relation.
The distribution relation plays an important role to resolve the ambiguity of integration constants.
\begin{Def}(Coleman Eisenstein-Kronecker series)
Let $m,b$ be integers with $b\geq 0$. We define 
the \textbf{Coleman Eisenstein-Kronecker series} $E_{m,b}^{\operatorname{col}}$ on $U:=E(\C_{p})\setminus [0]$ recursively as follows:
\begin{itemize}
\item [i)] $E_{0,b}^{\operatorname{col}}:=(-1)^{b-1}F_{b}^{\operatorname{col}}$.\\
This function satisfies the distribution relation by the formula (\ref{distribution relation of Coleman function of generating function}).
\item [ii)] If $m>0$, we define $E_{m,b}^{\operatorname{col}}$ by the Coleman function 
$$
E_{m,b}^{\operatorname{col}}:=-\int E_{m-1,b}^{\operatorname{col}}\omega
$$
with a constant term normalized by satisfying the distribution relation
\begin{equation}
\sum_{z_{\alpha}\in E[\alpha]}E_{m,b}^{\operatorname{col}}(z+z_{\alpha})=\alpha^{1-m}\overline{\alpha}^{1-b}E_{m,b}^{\operatorname{col}}(\alpha z)
\qquad \text{for any}\ 0\neq \alpha\in \mathcal{O}_{K}.
\label{distribution relation of p-adic Eisenstein-Kronecker series}
\end{equation}
\item [iii)] If $m<0$, we define $E_{m,b}^{\operatorname{col}}$ by 
$$
dE_{m+1,b}^{\operatorname{col}}:=-E_{m,b}^{\operatorname{col}}\omega,
$$
\end{itemize}
where $\omega$ is the invariant differential of the elliptic curve.
\label{Definition of p-adic Eisenstein-Kronecker series}
\end{Def}
There exists uniquely such a Coleman function $E_{m,b}^{\operatorname{col}}$ on $U$ defined
by the iterated integration $E_{m+1,b}^{\operatorname{col}}:=-\int E_{m,b}^{\operatorname{col}}\omega$
satisfying the distribution relation
$$
\sum_{z_{\alpha}\in E[\alpha]}E_{m+1,b}^{\operatorname{col}}(z+z_{\alpha})=\alpha^{-m}\overline{\alpha}^{1-b}E_{m+1,b}^{\operatorname{col}}(\alpha z)
\qquad \text{for any}\ 0\neq \alpha\in \mathcal{O}_{K}.
$$
The existence and uniqueness are insured by (\cite{BFK} Proposition 3.9).
\begin{Rmk}
The distribution relation (\ref{distribution relation of p-adic Eisenstein-Kronecker series}) is the $p$-adic analogue of 
the distribution relation of the classical complex Eisenstein-Kronecker series
$$
\sum_{z_{\alpha}\in E[\alpha]}E_{m,b}(z+z_{\alpha})=\alpha^{1-m}\overline{\alpha}^{1-b}E_{m,b}(\alpha z)
\qquad \text{for any}\ 0\neq \alpha\in \mathcal{O}_{K}.
$$
We can show this by using the following orthogonality of character
$$
\sum_{z_{\alpha}\in E[\alpha]} \exp\left(\frac{z_{\alpha}\overline{\gamma}-\overline{z_{\alpha}}\gamma}{A}\right)=
\begin{cases}
N(\alpha)(=\alpha\overline{\alpha}) & \text{if}\ \gamma\in \overline{\alpha}\Gamma\\
0 & \text{if}\ \gamma\not\in \overline{\alpha}\Gamma
\end{cases}
$$
\end{Rmk}
The construction of $p$-adic Eisenstein-Kronecker series in Definition \ref{Definition of p-adic Eisenstein-Kronecker series}
allows us to choose constant term when $m>0$.
By the convergence property of $F_{1}$ in (\ref{BKT Proposition 2.16}), $E_{m,1}^{\operatorname{col}}$ is defined at any point in $U:=E(\C_{p})\setminus[0]$,
and if $b>1$ then $E_{m,b}^{\operatorname{col}}$ is defined on $E(\mathbb{C}_{p})$.
When $b=0$, since $F_0=1$ and the definition of  $E_{0,b}^{\operatorname{col}}$, we have $E_{0,0}^{\operatorname{col}}=-F_{0}^{\operatorname{col}}=-1$.
This show that we have $E_{a,0}^{\operatorname{col}}=0$ for $a<0$.
Note that the values of $E_{m, b}^{\operatorname{col}} (z)$ are independent of the choice of the branch of the $p$-adic logarithm (see \cite{BFK} Lemma 3.12).\\
\quad According to (\cite{BFK} Proposition 3.11), the $p$-adic Eisenstein-Kronecker series $E_{m,b}^{\operatorname{col}}$ interpolates the classical complex
Eisenstein-Kronecker series $E_{m,b}$ for $m\leq 0$.\\
\quad For $m,b$ be integers with $b\geq 0$, we define
\begin{equation}
E_{m,b}^{(p)}(z):=E_{m,b}^{\operatorname{col}}(z)-\frac{1}{\pi^{m}\overline{\pi}^{b}}E_{m,b}^{\operatorname{col}}(\pi z).
\label{twisted p-adic Eisenstein-Kronecker series as Coleman function}
\end{equation}
\section{Main result}
In previous results, the construction of the $p$-adic measure interpolating special values of algebraic Hecke characters was established by Manin and Vi\v{s}ik
\cite{MaVi}, N. Katz \cite{Katz2}, R. I. Yager \cite{Yag}, and de Shalit \cite{dS} etc.
In \cite{BKT}, when the conductor of an algebraic Hecke character is prime to $p$,  K. Bannai, S. Kobayashi, T. Tsuji
related $p$-adic Eisenstein-Kronecker numbers
and non-critical values of the $p$-adic $L$-function associated to algebraic Hecke characters
under the another construction of the measure with the Kronecker theta function.\\
\quad  In this paper, by the method of $p$-adic analogue as Coleman functions, 
we related $p$-adic Eisenstein-Kronecker series and the non-critical values of $p$-adic $L$-function 
of an imaginary quadratic field with class number 1
associated to the algebraic 
Hecke character whose conductor is divisible by $p$. 

\subsection{Preludes for main results}\leavevmode\\
We keep the notation in \S0.
The $p$-adic $L$-function interpolates the classical $L$-function at critical points in the meaning of Deligne \cite{De}.
Let $\varphi: I(\mathfrak{g})\to \overline{K}^{\times}$ be an algebraic Hecke character of infinite type $(m,n)\in \Z^{2}$ whose conductor divides $\mathfrak{g}$
and $\varphi_{p}:\X\to \C_{p}^{\times}$ be a $p$-adic character. We fix a pair
$(\Omega, \Omega_{\mathfrak{p}})\in \C^{\times}\times \mathcal{O}_{\C_{p}}^{\times}$ of a complex period and a $p$-adic period.
There exists a $p$-adic function $L_{p}(\varphi_{p})$ such that 
$$
\frac{L_{p}(\varphi_{p})}{\Omega_{\mathfrak{p}}^{n-m}}=(-m-1)!\left(\frac{2\pi}{\sqrt{d_K}}\right)^n \left(1-\frac{\varphi^{-1}(\mathfrak{p})}{p}\right)
(1-\varphi(\overline{\mathfrak{p}}))\frac{L_{\mathfrak{g}}(0,\varphi)}{\Omega^{n-m}}
$$
for $m<0$ and $n\geq 0$,
both of which lies in $\overline{\Q}$ (see \cite{BKT} Proposition 2.26). We call $L_{p}(\varphi_{p})$ the \textbf{$p$-adic $L$-function}
 at the $p$-adic character $\varphi_{p}$.\\
\quad We give the construction of this $p$-adic $L$-function along \cite{BKT} \S2.4.
If $\widehat{E}$ is a formal group associated to $E\otimes_{\mathcal{O}_{K}}\mathcal{O}_{K_{\mathfrak{p}}}$
for $t=-2x/y$, $\widehat{E}$ is the Lubin-Tate formal group over $\mathcal{O}_{K_{\mathfrak{p}}}$.
Let a formal group law of $\widehat{E}$ be $\oplus$.
Let $\lambda: \widehat{E}\overset{\sim}{\longrightarrow} \widehat{\mathbb{G}}_{a}$ be a normalized formal logarithm by $\lambda'(0)=1$.
We have $\mathcal{O}_{K_{\mathfrak{p}}}$-linear isomorphisms
\begin{equation}
\operatorname{Hom}_{\mathcal{O}_{\mathbb{C}_{p}}}(\widehat{E}, \widehat{\G}_{m})\cong 
\operatorname{Hom}_{\Z_{p}}(T_{p}(\widehat{E}), T_{p}(\widehat{\G}_{m})) 
\overset{\cong}{\longleftarrow}\mathcal{O}_{K_{\mathfrak{p}}},
\label{isomorphism of Tate module}
\end{equation}
the first isomorphism holds by (\cite{Ta2} \S4.2 Corollary 1) and the second isomorphism holds 
since 
$
\operatorname{Hom}_{\Z_{p}}(T_{p}(\widehat{E}), T_{p}(\widehat{\G}_{m}))
$
is a free $\mathcal{O}_{K_{\mathfrak{p}}}$-module of rank 1 by
(\cite{Ta2} \S4.2 Proposition 12),
where 
$T_{p}(\widehat{E}):=\varprojlim_{n}\widehat{E}[p^n]$ (resp. $T_{p}(\widehat{\G}_{m}):=\varprojlim_{n}\widehat{\G}_{m}[p^n]$) is a Tate module
and $\widehat{\G}_{m}$ is a multiplicative formal group. $T_{p}(\widehat{\G}_{m})$ is a $\Z_{p}$-module of rank 1 (see \cite{Ta2} \S2.1 Examples (b)).
Then there exists a homomorphism $\eta_{\mathfrak{p}}: \widehat{E}\to \widehat{\G}_{m}$ such that if $t\in \widehat{E}(\mathfrak{m}_{\C_{p}})[\pi]$ 
then $1+\eta_{\mathfrak{p}}(t)$
is a $p$-th power root of 1 and that the image of $ \widehat{E}(\mathfrak{m}_{\C_{p}})[\pi]$ does not contain power roots of 1. 
If $\eta_{\mathfrak{p}}(t)=\Omega_{\mathfrak{p}}^{-1}t+\cdots \in \mathcal{O}_{\C_{p}}[[t]]$, by the uniqueness of  the formal logarithm $\lambda$,
we have the commutative diagram:
$$
\xymatrix@d@C=15mm@R=30mm{
\widehat{E}\ar@{->}[r]_{\lambda}^{\cong} \ar@{->}[dr]^{}_{\eta_{\mathfrak{p}}}
&\widehat{\mathbb{G}}_{a}\\
&\widehat{\mathbb{G}}_{m}\ar@{->}[u]^{\cong}_{w\mapsto\Omega_{\mathfrak{p}}\log(1+w)} 
}
$$
where the isomorphism $\widehat{\G}_{a}\overset{\sim}{\longrightarrow}\widehat{\G}_{m}$ follows from that a characteristic of $\mathcal{O}_{\C_{p}}$ is 0.\\
\quad Since $E$ is ordinary, $\widehat{E}$ has a height 1 (\cite{Sil1} V. Theorem 3.1 (b)). Then 
$\eta_{\mathfrak{p}}: \widehat{E}\to \widehat{\G}_{m}$ is isomorphism (see \cite{Lu} \S4. Corollary 4.3.3.).
Then we can take the isomorphism $\eta_{\mathfrak{p}}: \widehat{E}\overset{\cong}{\longrightarrow} \widehat{\G}_{m}$ over $\mathcal{O}_{\C_{p}}$ by
\begin{equation}
\eta_{\mathfrak{p}}(t)=\exp\left(\frac{\lambda(t)}{\Omega_{\mathfrak{p}}}\right)-1.
\label{isomorphism between formal groups}
\end{equation}
We call $\Omega_{\mathfrak{p}}\in \mathcal{O}_{\C_{p}}^\times$ a \textbf{$p$-adic period}, which is regarded as the $p$-adic analogue of the complex period $\Omega$.
The second isomorphism of (\ref{isomorphism of Tate module}) is given by associating to any $x\in \mathcal{O}_{K_{\mathfrak{p}}}$ the homomorphism 
of formal groups defined by $\exp(x\lambda(s)/\Omega_{\mathfrak{p}})$, and depends on the choice of $\Omega_{\mathfrak{p}}$.
In addition, since $E$ is ordinary, we have $(p)=\mathfrak{p}\overline{\mathfrak{p}}$. 
Therefore we have $\mathcal{O}_{K_{\mathfrak{p}}}\cong \Z_{p}$.\\
\quad Now we fix a natural number $N$. Let $s_{N}\in \widehat{E}(\mathfrak{m}_{\mathbb{C}_{p}})[\pi^{N}]$ be any $\pi^N$-torsion point.
Let $z_{N}$ be the image of $s_{N}$ by an inclusion map
$\widehat{E}(\mathfrak{m}_{\mathbb{C}_{p}})[\pi^{N}]\hookrightarrow E(\overline{\mathbb{Q}})\hookrightarrow E(\mathbb{C})\cong \mathbb{C}/\Gamma$.
Then $z_{N}$ is a $\pi^N$-torsion point in $E(\mathbb{C})[\pi^N]$.
Now we put $\displaystyle{z_{N}:=\Omega/\pi^N\in \pi^{-N}\Gamma/\Gamma\cong E(\mathbb{C})[\pi^N]}$.
Then we can take an isomorphism 
$\eta_{\mathfrak{p}}: \widehat{E}(\mathfrak{m}_{\mathbb{C}_{p}})[\pi^N]\overset{\sim}{\to} \widehat{\mathbb{G}}_{m}(\mathfrak{m}_{\C_{p}})[p^N]$ 
of formal groups over $\mathcal{O}_{\C_{p}}$ satisfying 
\begin{equation}
1+\eta_{\mathfrak{p}}(s_{N})=\chi_{z_N}(\Omega)
:=\exp\left(\frac{\Omega \overline{z}_{N}-\overline{\Omega}z_{N}}{A(\Gamma)} \right).
\label{choice of isomorphism of Tate module}
\end{equation}
This $\eta_{\mathfrak{p}}$ is induced by $\eta_{\mathfrak{p}}(t)=\exp(\lambda(t)/\Omega_{\mathfrak{p}})-1$.\\
\quad Let $W$ be the integer ring of the completion field at $p$ of a maximal unramified extension of $\mathbb{Q}_{p}$.
 According to (\cite{BK} Corollary 2.18),
the Kronecker theta function has the property of $p$-adic integrality as follows:
For torsion points $z_0,w_0\in E(\overline{K})_{\operatorname{tors}}$ with orders prime to $\mathfrak{p}$, 
we have
$$
\widehat{\Theta}_{z_0,w_0}(s,t):=\Theta_{z_0,w_0}(z,w)|_{z=\lambda(s),w=\lambda(t)}\in W[s^{-1},t^{-1}][[s,t]].
$$
In particular, 
$$
\widehat{\Theta}^{*}_{z_0,w_0}(s,t):=\widehat{\Theta}_{z_0,w_0}(s,t)-\chi_{z_0}(w_0)\delta_{z_0}s^{-1}-\delta_{w_0}t^{-1}\in W[[s,t]],
$$
where $\delta_{x}=1$ if $x\in \Gamma$ and $\delta_{x}=0$ otherwise.
Therefore for non-zero torsion points $z_0,w_0\in E(\overline{K})_{\operatorname{tors}}$ with orders prime to $\mathfrak{p}$,
we can characterize a $p$-adic measure $\mu_{z_0,w_0}: \mathbb{Z}_{p}\times \mathbb{Z}_{p}\to W$ by
\begin{equation}
\int_{\mathbb{Z}_{p}\times \mathbb{Z}_{p}}(1+s)^x (1+t)^y d\mu_{z_0,w_0}(x,y)=\widehat{\Theta}_{z_0,w_0}^{*}(s,t).
\label{construction of two variable p-adic measure of Kronecker theta function}
\end{equation}
In particular, if $z_0,w_0\not\in \Gamma$, we have $\widehat{\Theta}_{z_0,w_0}^{*}(s,t)=\widehat{\Theta}_{z_0,w_0}(s,t)$.
The existence and uniqueness of this $p$-adic measure $\mu_{z_0,w_0}: \mathbb{Z}_{p}\times \mathbb{Z}_{p}\to W$ follows from
$\widehat{\Theta}^{*}_{z_0,w_0}(s,t)\in W[[s,t]]$ (see \cite{Yag} \S6).
By the isomorphism
$\eta_{\mathfrak{p}}:\ \widehat{E}\overset{\sim}{\longrightarrow}\widehat{\G}_{m}$, 
we can rewrite the formula (\ref{construction of two variable p-adic measure of Kronecker theta function}) as
\begin{equation}
\int_{\mathbb{Z}_{p}\times \mathbb{Z}_{p}}\exp\left(\frac{x\lambda(s)}{\Omega_{\mathfrak{p}}}\right)
\exp\left(\frac{y\lambda(t)}{\Omega_{\mathfrak{p}}}\right)d\mu_{z_0,w_0}(x,y)=
\widehat{\Theta}_{z_0,w_0}^{*}(s,t).
\label{the measure by Kronecker theta function}
\end{equation}
\medskip
For a non-zero torsion point $z_{0}\in E(\overline{K})_{\operatorname{tors}}$ with an order prime to $p$ and $g$ as above,
 we define a variant measure $\mu_{z_{0}, 0}^{(g)}$ on $\Z_{p}\times \Z_{p}$ of the measure $\mu_{z_0, 0}$ in (\ref{the measure by Kronecker theta function})
 by the formula
\begin{equation}
\int_{\Z_{p}\times\Z_{p}}\exp\left(\frac{x\lambda(s)}{\Omega_{\mathfrak{p}}}\right)
\exp\left(\frac{y\lambda(t)}{\Omega_{\mathfrak{p}}}\right)d\mu_{z_0, 0}^{(g)}(x,y)=
\widehat{\Theta}_{z_0, 0}^{*}([g^{-1}]s, [\overline{g}]t)
\label{twisted of g of the measure by Kronecker theta function}
\end{equation}
\quad For $\alpha_0\in (\mathcal{O}_{K}/\mathfrak{g})^{\times}$, 
let $\alpha_{0}\Omega/g \in \mathfrak{g}^{-1}\Gamma/\Gamma\cong E(\C)[\mathfrak{g}]$ be a primitive $\mathfrak{g}$-torsion point 
and $\mu_{\alpha_{0}\Omega/g, 0}^{(g)}$
induces a measure on $(\mathcal{O}_{K}\otimes_{\Z}\Z_{p})^{\times}$ through the isomorphism (\ref{the projections to the first and second factors}).
Similarly to (\cite{BKT} \S2.4), for a continuous function $f: \X \longrightarrow \C_{p}$, 
we define the measure $\mu_{g}$ on $\X$ by
\begin{equation}
\int_{\X}f(\alpha)d\mu_{g}(\alpha):=g^{-1}\Omega_{\mathfrak{p}}\sum_{\alpha_{0}\in (\mathcal{O}_{K}/\mathfrak{g})^{\times}}
\int_{(\mathcal{O}_{K}\otimes_{\Z}\Z_{p})^{\times}}f(\alpha_{0}\alpha)\kappa_{1}(\alpha)d\mu_{\alpha_{0}\Omega/g,0}^{(g)}(\alpha).
\label{definition of the measure on X}
\end{equation}
\begin{Def}[$p$-adic $L$-function]
Let $\mathfrak{f}$ be an integral ideal of $\mathcal{O}_{K}$ prime to $p$.
The $p$-adic $L$-function of $K$ is the function whose domain is the set of all $p$-adic continuous characters $\phi_{p}: \X \to \C_{p}^{\times}$ on $\X$,
and is defined at the character $\phi_{p}$ by 
\begin{equation}
L_{p}(\phi_{p}):=\int_{\X}\phi_{p}(\alpha)d\mu_{g}(\alpha).
\label{definition of p-adic L-function}
\end{equation}
Here the measure $\mu_{g}$ corresponds to the measure denoted by $\mu$ in (\cite{dS} Theorem 4.14) 
through the canonical isomorphism $\operatorname{Gal}(K(\mathfrak{g}p^{\infty})/K)\cong \X$.
\end{Def}
By using the formula (\ref{definition of the measure on X}), we can rewrite the $p$-adic $L$-function (\ref{definition of p-adic L-function}) 
at the $p$-adic character $\varphi_{p}$ in \S0 as
\begin{equation}
L_{p}(\varphi_{p})=g^{-1}\Omega_{\mathfrak{p}}\sum_{\alpha_{0}\in (\mathcal{O}_{K}/\mathfrak{g})^{\times}}\chi_{\mathfrak{g}}(\alpha_{0})
\int_{(\mathcal{O}_{K}\otimes_{\Z}\Z_{p})^{\times}}\chi_{1}(\alpha)\chi_{2}(\alpha)\kappa_{1}(\alpha)^{m+1}\kappa_{2}(\alpha)^{n}
d\mu_{\alpha_{0}\Omega/g, 0}^{(g)}(\alpha),
\label{redefinition of p-adic L-function 1}
\end{equation}
where $\chi_{1}, \chi_{2}$ in (\ref{redefinition of p-adic L-function 1}) are characters in \S0.
In addition by the following isomorphism
$$
\mathcal{O}_{K_{\mathfrak{p}}}^{\times}\times \mathcal{O}_{K_{\overline{\mathfrak{p}}}}^{\times}
\overset{\cong}{\longrightarrow}\Z_{p}^{\times}\times \Z_{p}^{\times}\quad
(\alpha_{1}, \alpha_{2})\overset{\sim}{\longmapsto}(\alpha_{1}^{-1}, \alpha_{2}),
$$
and putting $x:=\kappa_{1}(\alpha)^{-1}$ and $y:=\kappa_{2}(\alpha)$, we rewrite the formula (\ref{redefinition of p-adic L-function 1}) as 
\begin{equation}
L_{p}(\varphi_{p})=g^{-1}\Omega_{\mathfrak{p}}\sum_{\alpha_{0}\in (\mathcal{O}_{K}/\mathfrak{g})^{\times}}\chi_{\mathfrak{g}}(\alpha_{0})
\int_{\Z_{p}^{\times}\times \Z_{p}^{\times}}\chi_{1}(x)^{-1}\chi_{2}(y)x^{-m-1}y^{n}
d\mu_{\alpha_{0}\Omega/g, 0}^{(g)}(x,y).
\label{redefinition of p-adic L-function 2}
\end{equation}
\quad Since $p=\pi\bar{\pi}$ and $\chi_1$ (resp. $\chi_2$) is defined by extending to 0 at all values not prime to $\pi$ (resp. $\overline{\pi}$),
$\chi_1$ (resp. $\chi_2$)  is 0 on $p\mathbb{Z}_{p}$.
Therefore if we replace the integration domain of (\ref{redefinition of p-adic L-function 2}) by 
$\mathbb{Z}_{p}\times\mathbb{Z}_{p}$, $\mathbb{Z}_{p}^{\times}\times\mathbb{Z}_{p}$, $\mathbb{Z}_{p}\times\mathbb{Z}_{p}^{\times}$, or
$\mathbb{Z}_{p}^{\times}\times\mathbb{Z}_{p}^{\times}$, all integrations of 
the right side of the formula (\ref{redefinition of p-adic L-function 2}) have the same value on 
$\mathbb{Z}_{p}\times\mathbb{Z}_{p}$, $\mathbb{Z}_{p}^{\times}\times\mathbb{Z}_{p}$, $\mathbb{Z}_{p}\times\mathbb{Z}_{p}^{\times}$, or
$\mathbb{Z}_{p}^{\times}\times\mathbb{Z}_{p}^{\times}$.
\begin{Def}[Gauss sum]\leavevmode
\begin{itemize}
\item[a)]
For $x\in \mathcal{O}_{K}/(\pi^N)$, we define \textbf{the Gauss sum} on $\mathcal{O}_{K}/(\pi^N)$ by
$$
\tau(\chi_1,x):=\sum_{u\in (\mathcal{O}_{K}/(\pi^N))^{\times}}\overline{\chi_{1}(u)}\langle \Omega, xuz_N\rangle
=\sum_{u\in (\mathcal{O}_{K}/(\pi^N))^{\times}}\overline{\chi_{1}(u)}\exp\left(\frac{\Omega \overline{xuz_{N}}-\overline{\Omega}xuz_{N}}{A}\right),
$$
where $z_N:=\Omega/\pi^N\in \pi^{-N}\Gamma/\Gamma\cong E(\mathbb{C})[\pi^N]$ is a $\pi^N$-torsion point.
\item[b)]
For $y\in \mathcal{O}_{K}/(\bar{\pi}^N)$, we define the Gauss sum on $\mathcal{O}_{K}/(\bar{\pi}^N)$ by
$$
\tau(\chi_2,y):=\sum_{v\in (\mathcal{O}_{K}/(\bar{\pi}^N))^{\times}}\overline{\chi_{2}(v)}\langle \Omega, \bar{y}\bar{v}w_N\rangle
=\sum_{v\in (\mathcal{O}_{K}/(\bar{\pi}^N))^{\times}}\overline{\chi_{2}(v)}\exp\left(\frac{\Omega \overline{\bar{y}\bar{v}w_{N}}
-\overline{\Omega}\bar{y}\bar{v}w_{N}}{A}\right),
$$
where $w_N:=\Omega/\pi^N\in \pi^{-N}\Gamma/\Gamma\cong E(\mathbb{C})[\pi^N]$ is a $\pi^N$-torsion point.
\end{itemize}
\label{Definition of Gauss sum}
\end{Def}

\begin{Lem}[Properties of Gauss sum] Under the assumption of Definition \ref{Definition of Gauss sum}, we have
\begin{itemize}
\item[a)]
$\tau(\chi_1,x)=\chi_1(x)\tau(\chi_1)$, where $\tau(\chi_1):=\tau(\chi_1,1)$.
In addition, if $x$ is not prime to $\pi$, we have $\tau(\chi_{1}, x)=0$, and 
if $x$ is prime to $\pi$,  we have $|\tau(\chi_{1},x)|=\sqrt{N(\mathfrak{p}^N)}=\sqrt{p^N}$.
\item[b)]
$\tau(\chi_2,y)=\chi_2(y)\tau(\chi_2)$, where $\tau(\chi_2):=\tau(\chi_2,1)$.
In addition, if $y$ is not prime to $\bar{\pi}$, we have $\tau(\chi_{2}, y)=0$ and 
if $y$ is prime to $\bar{\pi}$, we have $|\tau(\chi_{2},y)|=\sqrt{N(\overline{\mathfrak{p}}^N)}=\sqrt{p^N}$.
\end{itemize}
\label{Properties of Gauss sum}
\end{Lem}
\begin{proof}
For the proof of b), if we replace $y$ for $\bar{y}$ and $v$ for $\bar{v}$, we can reduce to a).
So we have only to prove a). \\
\quad If $x$ is prime to $\pi$, since $(\mathcal{O}_{K}/(\pi^N))^{\times}\to (\mathcal{O}_{K}/(\pi^N))^{\times}, u\mapsto xu$ is bijective
and $\chi_{1}^{-1}=\overline{\chi_{1}}$ because $\chi_{1}$ is a finite character, then we have
$$
\tau(\chi_{1}, x)=\chi_{1}(x)\tau(\chi_{1}).
$$
\quad Next we consider when $x$ is not prime to $\pi$. Let $d\neq 1$ be a greatest common divisor of $x$ and $\pi$.
Since $\chi_{1}$ is primitive, there exists $a\in (\mathcal{O}_{K}/(\pi^N))^{\times}$ such that 
$$
\chi_{1}(a)\neq 1 \quad \text{and}\quad  a\equiv 1\operatorname{mod} \frac{\pi^N}{d}.
$$
So we have $xa\equiv x\ (\operatorname{mod} \pi^N)$. Therefore we have
$$
\overline{\chi_{1}(a)}\tau(\chi_{1},x)=\tau(\chi_{1},x).
$$
Since $\chi_{1}(a)\neq 1$, we have $\tau(\chi_{1},x)=0$. On the other hand, since $\chi_{1}(x)=0$, we have
$\chi_{1}(x)\tau(\chi_{1})=0.$
Hence $\tau(\chi_{1},x)=\chi_{1}(x)\tau(\chi_{1})$.\\
\quad We show the latter part of a).
Since $\Gamma=\Omega\mathcal{O}_{K}$, we have $A:=A(\Gamma)=\sqrt{d_{K}}\Omega\overline{\Omega}/2\varpi$,
where $d_{K}$ is a discriminant of $K$. Then
\begin{align*}
\tau(\chi_{1},x)=\sum_{u\in (\mathcal{O}_{K}/(\pi^N))^{\times}}\overline{\chi_{1}(u)}\exp\left[
\frac{2\varpi}{\sqrt{d_{K}}}\left( \overline{\left(\frac{xu}{\pi^N} \right)}-\frac{xu}{\pi^N} \right)
\right].
\end{align*}
Now for $z\in \mathcal{O}_{K}/(\pi^N)$, we have
\begin{equation}
\tau(\chi_{1}, z)\overline{\tau(\chi_{1},z)}=\sum_{u,v\in  (\mathcal{O}_{K}/(\pi^N))^{\times}}
\overline{\chi_{1}(u)}\chi_{1}(v)\exp\left[ 
\frac{2\varpi}{\sqrt{d_{K}}}\left(\frac{\bar{z}}{\bar{\pi}^{N}}(\bar{u}-\bar{v})-\frac{z}{\pi^N}(u-v)\right)
\right].
\label{on the way to calculation of Gauss sum}
\end{equation}
If we take the sum of both sides of (\ref{on the way to calculation of Gauss sum}) over $z\in \mathcal{O}_{K}/(\pi^N)$,
since $\tau(\chi_{1},z)=0$ holds if $(z,\pi)\neq 1$, so
the sum over $z\in \mathcal{O}_{K}/(\pi^N)$ of the left side of (\ref{on the way to calculation of Gauss sum})
and the sum over $z\in (\mathcal{O}_{K}/(\pi^N))^{\times}$ have the same value. Then  
giving the sum over $z\in (\mathcal{O}_{K}/(\pi^N))^{\times}$ to the left side of (\ref{on the way to calculation of Gauss sum}), we have
$$
\sum_{z\in (\mathcal{O}_{K}/(\pi^N))^{\times}}\tau(\chi_{1}, z)\overline{\tau(\chi_{1}, z)}=
\sum_{u, v\in (\mathcal{O}_{K}/(\pi^N))^{\times}}\chi_{1}(u)\overline{\chi_{1}(v)}
\sum_{z\in (\mathcal{O}_{K}/(\pi^N))^{\times}}\langle \Omega, zuz_N \rangle \langle \overline{\Omega}, \bar{z}\bar{u}\overline{z_N} \rangle
$$
Since $z$ is prime to $\pi$, $(\mathcal{O}_{K}/(\pi^N))^{\times}\to (\mathcal{O}_{K}/(\pi^N))^{\times},\quad u\mapsto z^{-1}u, v\mapsto z^{-1}v$ is bijective, we have
\begin{align*}
\sum_{z\in (\mathcal{O}_{K}/(\pi^N))^{\times}}\tau(\chi_{1}, z)\overline{\tau(\chi_{1}, z)}
&=\varepsilon(\pi^N)
\sum_{u, v\in (\mathcal{O}_{K}/(\pi^N))^{\times}}\chi_{1}(u)\overline{\chi_{1}(v)}\langle \Omega, uz_N \rangle \langle \overline{\Omega}, \bar{u}\overline{z_N} \rangle&\\
&=\varepsilon(\pi^N)|\tau(\chi_{1})|^2=\varepsilon(\pi^N)|\tau(\chi_{1}, x)|^2,&
\end{align*}
where $\varepsilon(\pi^N):=\#(\mathcal{O}_{K}/(\pi^N))^{\times}$ is an Euler function and we used $\chi_{1}(z)\overline{\chi_{1}(z)}=|\chi_{1}(z)|^2=1$.
Next we give the sum over $z\in (\mathcal{O}_{K}/(\pi^N))^{\times}$ to the right side of (\ref{on the way to calculation of Gauss sum}). 
Since 
$$
\sum_{z\in (\mathcal{O}_{K}/(\pi^N))^{\times}}\exp\left[ 
\frac{2\varpi}{\sqrt{d_{K}}}\left(\frac{\bar{z}}{\bar{\pi}^{N}}(\bar{u}-\bar{v})-\frac{z}{\pi^N}(u-v)\right)
\right]=
\begin{cases}
\varepsilon(\pi^N) N(\mathfrak{p}^N)& \text{if}\ u\equiv v\ (\operatorname{mod}\pi^N)\\
0& \text{otherwise}
\end{cases}
$$
the sum over $z\in (\mathcal{O}_{K}/(\pi^N))^{\times}$ of the right side of (\ref{on the way to calculation of Gauss sum}) is $\varepsilon(\pi^{N})p^N$.
Therefore we have 
$$
|\tau(\chi_{1}, x)|^2 = N(\mathfrak{p}^N)=p^N.
$$
\end{proof}
By Lemma \ref{Properties of Gauss sum} and using the formula (\ref{choice of isomorphism of Tate module}),
we can rewrite the $p$-adic $L$-function (\ref{redefinition of p-adic L-function 2}) at the $p$-adic character $\varphi_{p}$
as 
\begin{align}
\label{redefinition of p-adic L-function 3}
L_{p}(\varphi_{p})=
&\frac{g^{-1}\Omega_{\mathfrak{p}}}{\tau(\overline{\chi_{1}})\tau(\chi_2)}\sum_{\alpha_0\in (\mathcal{O}_{K}/\mathfrak{g})^{\times}}
\chi_{\mathfrak{g}}(\alpha_0)
\sum_{u\in (\mathcal{O}_{K}/(\pi^N))^{\times}}\sum_{v\in (\mathcal{O}_{K}/(\bar{\pi}^N))^{\times}}\chi_1(u)\overline{\chi_2(v)}\times&\\
&\qquad \int_{\mathbb{Z}_{p}\times\mathbb{Z}_{p}}\exp\left(\frac{x\lambda([u]s_N)}{\Omega_{\mathfrak{p}}}\right)\exp\left(\frac{y\lambda([\bar{v}]t_N)}{\Omega_{\mathfrak{p}}}\right)
x^{-m-1}y^n d\mu_{\alpha_0\Omega/g,0}^{(g)}(x,y)\notag.&
\end{align}
where $x\in \mathbb{Z}_{p}$ (resp. $y\in \mathbb{Z}_{p}$) in the formula (\ref{redefinition of p-adic L-function 3}) is an element
obtained by
the lifting via the natural projection
$\mathbb{Z}_{p}\cong \mathcal{O}_{K_{\mathfrak{p}}}\twoheadrightarrow \mathcal{O}_{K}/(\pi^N)$
(resp. $\mathbb{Z}_{p}\cong \mathcal{O}_{K_{\overline{\mathfrak{p}}}}\twoheadrightarrow \mathcal{O}_{K}/(\bar{\pi}^N)$) 
 and $s_N\in \widehat{E}(\mathfrak{m}_{\mathbb{C}_{p}})[\pi^N]$ (resp. $t_N\in \widehat{E}(\mathfrak{m}_{\mathbb{C}_{p}})[\pi^N]$)
 is the $\pi^N$-torsion point in the formal group
 corresponding to the $\pi^N$-torsion point 
 $z_N=\Omega/\pi^N\in \pi^{-N}\Gamma/\Gamma\cong E(\mathbb{C})[\pi^N]$ 
 (resp. $w_N=\Omega/\pi^N\in \pi^{-N}\Gamma/\Gamma\cong E(\mathbb{C})[\pi^N]$)
through the inclusion map 
$\widehat{E}(\mathfrak{m}_{\mathbb{C}_{p}})[\pi^N]\hookrightarrow E(\overline{\mathbb{Q}})[\pi^N] \hookrightarrow \mathbb{C}/\Gamma$.\\
\quad Since the $p$-adic $L$-function (\ref{redefinition of p-adic L-function 3}) is rewritten by using Lemma \ref{Properties of Gauss sum},
similarly to (\ref{redefinition of p-adic L-function 2}), the $p$-adic $L$-function (\ref{redefinition of p-adic L-function 3}) has the same value
on $\Z_{p}\times \Z_{p}$, $\Z_{p}^{\times}\times \Z_{p}$, $\Z_{p}\times \Z_{p}^{\times}$, or $\Z_{p}^{\times}\times \Z_{p}^{\times}$.

\subsection{The non-critical values of the $p$-adic Hecke $L$-function whose conductor is divisible by $p$ 
 and $p$-adic Eisenstein-Kronecker series}\leavevmode\\

We calculate the rewritten $p$-adic $L$-function (\ref{redefinition of p-adic L-function 3}), 
focusing on each summand of the $p$-adic $L$-function. 
The method of calculation of the $p$-adic $L$-function (\ref{redefinition of p-adic L-function 3})
is as follows:
We first rewrite the $p$-adic $L$-function (\ref{redefinition of p-adic L-function 3}) constructed using the two-variable measure in terms of the one-variable measure
(see Proposition \ref{redefinition of immediate p-adic L-function 1}).
Then we express the integration in Proposition \ref{redefinition of immediate p-adic L-function 1} using the formula
(\ref{twisted p-adic Eisenstein-Kronecker series as Coleman function}) (see Lemma \ref{expressed as a Coleman function}). 
Finally, by using Lemma \ref{cancel}, we express the $p$-adic $L$-function (\ref{redefinition of p-adic L-function 3})
using the Coleman Eisenstein-Kronecker series (see Proposition \ref{redefinition of twisted p-adic L-function}).
\begin{Lem}
For any $s,t\in \widehat{E}(\mathfrak{m}_{\C_{p}})$, we have
\begin{align}
\label{pick up the rewritten p-adic L-function 3}
\sum_{v\in (\mathcal{O}_{K}/(\overline{\pi}^N))^{\times}}\overline{\chi_{2}(v)}
&\int_{\mathbb{Z}_{p}\times\mathbb{Z}_{p}}\exp\left(\frac{x\lambda(s\oplus[u]s_N)}{\Omega_{\mathfrak{p}}}\right)
\exp\left(\frac{y\lambda(t\oplus[\bar{v}]t_N)}{\Omega_{\mathfrak{p}}}\right)\mu_{\alpha_0\Omega/g,0}^{(g)}(x,y)& \\
&=\frac{\pi^N}{\tau(\overline{\chi_{2}})}\sum_{\gamma\in (\mathcal{O}_{K}/(\overline{\pi}^N))^{\times}}
\chi_{2}(\gamma)\widehat{\Theta}^{*}_{(\alpha_0 \Omega/g-\gamma \Omega/g)/\overline{\pi}^N, 0}([g^{-1}\overline{\pi}^{-N}](s\oplus [u]s_N), [\pi^N \overline{g}]t). \notag
\end{align}
\label{some portions of the $p$-adic $L$-function}
\end{Lem}
\begin{proof}
\begin{align*}
\text{(Left side)}
&=\sum_{v\in (\mathcal{O}_{K}/(\overline{\pi}^N))^{\times}}\overline{\chi_{2}(v)}
\widehat{\Theta}^{*}_{\alpha_0\Omega/g, \bar{g}\bar{v}w_N}([g^{-1}](s\oplus [u]s_N), [\overline{g}]t)&\\
&=\frac{1}{\tau(\overline{\chi_{2}})}\sum_{\gamma\in (\mathcal{O}_{K}/(\overline{\pi}^N))^{\times}}
\chi_{2}(\gamma)\sum_{v\in (\mathcal{O}_{K}/(\overline{\pi}^N))^{\times}}\langle \Omega, \bar{v}\bar{\gamma}w_{N}\rangle
\widehat{\Theta}^{*}_{\alpha_0\Omega/g, \bar{g}\bar{v}w_N}([g^{-1}](s\oplus [u]s_N), [\bar{g}]t).
\end{align*}
Here the first equation can be calculated by the formula (\ref{twisted of g of the measure by Kronecker theta function})
and the $p$-adic translation of the Kronecker theta function (see \cite{BK} Corollary 2.22) and the second equation can be calculated by
$\tau(\overline{\chi_{2}}, v)=\overline{\chi_{2}(v)}\tau(\overline{\chi_{2}})$ by Lemma \ref{Properties of Gauss sum} b).
The result of this lemma can be showed by the following lemma \ref{calculation of portion}.
\end{proof}
\begin{Lem}
\begin{align*}
\sum_{v\in (\mathcal{O}_{K}/(\overline{\pi}^N))^{\times}}\langle \Omega, \bar{v}\bar{\gamma}w_{N}\rangle
&\widehat{\Theta}^{*}_{\alpha_0\Omega/g, \bar{g}\bar{v}w_N}([g^{-1}](s\oplus [u]s_N), [\bar{g}]t)&\\
&=\sum_{w_N''\in \pi^{-N}\Gamma/\Gamma}\langle \gamma g^{-1}\Omega, w_{N}''\rangle
\widehat{\Theta}^{*}_{\alpha_0\Omega/g, w_N ''}([g^{-1}](s\oplus [u]s_N), [\overline{g}]t)
\end{align*}
\label{calculation of portion}
\end{Lem}
\begin{proof}
Since the formula (\ref{pick up the rewritten p-adic L-function 3}) is 0 when $v$ is not prime to $\overline{\pi}$, we have only consider the above formula when
$v$ is prime to $\overline{\pi}$.
Since $\mathcal{O}_{K}/(\overline{\pi}^N)\overset{\cong}{\longrightarrow}\pi^{-N}\Gamma/\Gamma, v\overset{\sim}{\longmapsto} \overline{v}w_N$
is bijective, if we put $w_N ':=\overline{v}w_N$, we have
\begin{align*}
\sum_{v\in (\mathcal{O}_{K}/(\overline{\pi}^N))^{\times}}\langle \Omega, \bar{v}\bar{\gamma}w_{N}\rangle
&\widehat{\Theta}^{*}_{\alpha_0\Omega/g, \bar{g}\bar{v}w_N}([g^{-1}](s\oplus [u]s_N), [\bar{g}]t)&\\
&=\sum_{w_N '\in \pi^{-N}\Gamma/\Gamma}\langle \gamma\Omega, w_N '\rangle
\widehat{\Theta}^{*}_{\alpha_0\Omega/g, \bar{g}w_N '}([g^{-1}](s\oplus [u]s_N), [\bar{g}]t)
\end{align*}
Since $g$ is prime to $p$ by the assumption, $g$ is prime to $\overline{\pi}$. Then 
$\pi^{-N}\Gamma/\Gamma\longrightarrow \pi^{-N}\Gamma/\Gamma, w_N '\longmapsto \overline{g}w_N '$ is bijective.
If we put $w_N '':=\overline{g}w_N '$, then we have
\begin{align*}
\sum_{v\in (\mathcal{O}_{K}/(\overline{\pi}^N))^{\times}}\langle \Omega, \bar{v}\bar{\gamma}w_{N}\rangle
&\widehat{\Theta}^{*}_{\alpha_0\Omega/g, \bar{g}\bar{v}w_N}([g^{-1}](s\oplus [u]s_N), [\bar{g}]t)&\\
&=\sum_{w_N''\in \pi^{-N}\Gamma/\Gamma}\langle \gamma g^{-1}\Omega, w_{N}''\rangle
\widehat{\Theta}^{*}_{\alpha_0\Omega/g, w_N ''}([g^{-1}](s\oplus [u]s_N), [\overline{g}]t)
\end{align*}
Since $g$ is prime to $p$, $g$ is prime to $\overline{\pi}$. Hence $\gamma g^{-1}\in (\mathcal{O}_{K}/(\overline{\pi}^{N}))^{\times}$.
By the lifting $\gamma g^{-1}$ into $\mathcal{O}_{K}$, we have $\gamma g^{-1}\Omega\in \mathcal{O}_{K}\Omega=\Gamma$.
So we can use the distribution relation of the Kronecker theta function (\ref{distribution relation of Kronecker theta function 1}).
Then we obtain this lemma.
\end{proof}
We show that the $p$-adic $L$-function (\ref{redefinition of p-adic L-function 3}) can be rewritten as follows:
\begin{Prop}
Let $m,n$ be any integers with $n\geq 0$. Then we have
\begin{align*}
L_{p}(\varphi_{p})=\frac{g^{-1}\Omega_{\mathfrak{p}}\pi^{N}}{\tau(\overline{\chi_{1}})\tau(\chi_{2})\tau(\overline{\chi_{2}})}
&\sum_{\alpha_{0}\in (\mathcal{O}_{K}/\mathfrak{g})^{\times}}\chi_{\mathfrak{g}}(\alpha_{0})\sum_{u\in (\mathcal{O}_{K}/(\pi^N))^{\times}}\chi_{1}(u)
\sum_{\gamma\in (\mathcal{O}_{K}/(\bar{\pi}^N))^{\times}}\chi_{2}(\gamma)&\\
&\int_{\Z_{p}^{\times}\times \Z_{p}}
\exp\left(\frac{x\lambda([g^{-1}\bar{\pi}^N u]s_N)}{\Omega_{\mathfrak{p}}}\right)x^{-m-1}y^n d\mu_{(\alpha_0 \Omega/g -\gamma\Omega/g)/\bar{\pi}^N, 0}(x,y)&
\end{align*}
\label{redefinition of immediate p-adic L-function 1}
\end{Prop}
\begin{proof}
Since $\alpha_{0}\Omega/g \in \mathfrak{g}^{-1}\Gamma/\Gamma$ and $\gamma g^{-1}\Omega \in\Gamma$, 
this value $(\alpha_0 \Omega/g-\gamma \Omega/g)/\overline{\pi}^N \in \mathfrak{g}^{-1} \bar{\pi}^{-N}\Gamma/\Gamma\cong E(\C)[\mathfrak{g}\overline{\pi}^N]$
is a $g\overline{\pi}^N$-torsion point. 
Since $(\alpha_0 \Omega/g-\gamma \Omega/g)/\overline{\pi}^N$ is a torsion point with an order prime to $\pi$,
by the formula (\ref{the measure by Kronecker theta function}), Lemma \ref{some portions of the $p$-adic $L$-function} becomes
\begin{align}
\label{pick up the p-adic L-function 3}
\int_{\Z_{p}\times \Z_{p}}\exp\left(\frac{x\lambda(s)}{\Omega_{\mathfrak{p}}}\right)&\exp\left(\frac{y\lambda(t)}{\Omega_{\mathfrak{p}}}\right)d\mu_{1}(x,y)\\
&=\int_{\Z_{p}\times \Z_{p}}\exp\left(\frac{x\lambda([g^{-1}\bar{\pi}^{-N}]s)}{\Omega_{\mathfrak{p}}}\right)
\exp\left(\frac{y\lambda([\pi^N \bar{g}]t)}{\Omega_{\mathfrak{p}}}\right)d\mu_{2}(x,y)\notag&
\end{align}
Here for simplicity, we put
\begin{align*}
d\mu_{1}(x,y)&:=\sum_{v\in (\mathcal{O}_{K}/(\pi^N))^{\times}}\overline{\chi_{2}}(v)\exp\left(\frac{x\lambda([u]s_N)}{\Omega_{\mathfrak{p}}}\right)
\exp\left(\frac{y\lambda([\overline{v}]t_N)}{\Omega_{\mathfrak{p}}}\right)d\mu_{\alpha_0 \Omega/g, 0}^{(g)}(x,y)&\\
d\mu_{2}(x,y)&:=\frac{\pi^N}{\tau(\overline{\chi_{2}})}\sum_{\gamma\in (\mathcal{O}_{K}/(\overline{\pi}^N))^{\times}}\chi_{2}(\gamma)
\exp\left(\frac{x\lambda([g^{-1}\bar{\pi}^{-N}u]s_N)}{\Omega_{\mathfrak{p}}}\right)d\mu_{(\alpha_0 \Omega/g-\gamma \Omega/g)/\bar{\pi}^N,0}(x,y).&
\end{align*}
By the isomorphism of formal groups $\eta_{\mathfrak{p}}: \widehat{E}(\mathfrak{m}_{\C_{p}})\longrightarrow \widehat{\G}_{m}(\mathfrak{m}_{\C_{p}})$
the formula (\ref{pick up the p-adic L-function 3}) is 
\begin{align*}
\int_{\Z_{p}\times \Z_{p}}&(1+\eta_{\mathfrak{p}}(s))^{x}(1+\eta_{\mathfrak{p}}(t))^{y}d\mu_{1}(x,y)&\\
&=\int_{\Z_{p}\times \Z_{p}}(1+\eta_{\mathfrak{p}}([g^{-1}\bar{\pi}^{-N}]s))^{x}(1+\eta_{\mathfrak{p}}([\pi^N \bar{g}]t))^{y}d\mu_{2}(x,y)&
\end{align*}
By the binomial expansion, we have
\begin{align*}
\sum_{m,n=0}^{\infty}\eta_{\mathfrak{p}}(s)^{m}\eta_{\mathfrak{p}}(t)^{n}&\int_{\Z_{p}\times \Z_{p}}\binom{x}{m}\binom{y}{n}d\mu_{1}(x,y)&\\
&=\sum_{m,n=0}^{\infty}\eta_{\mathfrak{p}}([g^{-1}\bar{\pi}^{-N}]s)^{m}\eta_{\mathfrak{p}}([\pi^N \overline{g}]t)^{n}\int_{\Z_{p}\times \Z_{p}}
\binom{x}{m}\binom{y}{n}d\mu_{2}(x,y)&
\end{align*}
Now we put
\begin{align*}
f(x,y)&:=\sum_{m,n=0}^{\infty}\eta_{\mathfrak{p}}(s)^{m}\eta_{\mathfrak{p}}(t)^{n}\binom{x}{m}\binom{y}{n},&\\
g(x,y)&:=\sum_{m,n=0}^{\infty}\eta_{\mathfrak{p}}([g^{-1}\bar{\pi}^{-N}]s)^{m}\eta_{\mathfrak{p}}([\pi^N \overline{g}]t)^{n}\binom{x}{m}\binom{y}{n}.&
\end{align*}
Since $|\eta_{\mathfrak{p}}(s)|<1$, $|\eta_{\mathfrak{p}}(t)|<1$ by $s,t \in \widehat{E}(\mathfrak{m}_{\C_{p}})$, 
then $\lim_{m,n\to \infty} \eta_{\mathfrak{p}}(s)^m \eta_{\mathfrak{p}}(t)^n=0$. 
By Mahler's theorem, $f(x,y)$ is a continuous function on $\Z_{p}\times \Z_{p}$.
On the other hand, since $g$ and $\overline{\pi}$ are prime to $\pi$, $[g^{-1}\bar{\pi}^{-N}]s \in \widehat{E}(\mathfrak{m}_{\C_{p}})$.
Then $\lim_{m,n\to \infty}\eta_{\mathfrak{p}}([g^{-1}\bar{\pi}^{-N}]s)^{m}\eta_{\mathfrak{p}}([\pi^N \overline{g}]t)^{n}=0$.
By Mahler's theorem, $g(x,y)$ is a continuous function on $\Z_{p}\times \Z_{p}$.
By arbitrariness of $s,t \in \widehat{E}(\mathfrak{m}_{\C_{p}})$, $f(x,y)$ and $g(x,y)$ are arbitrary continuous function on $\Z_{p}\times \Z_{p}$.
Since we pick up a part of the $p$-adic $L$-function (\ref{redefinition of p-adic L-function 3}), we can regard $f(x,y)$ and $g(x,y)$ as continuous functions of $(x,y)$
on $\Z_{p}^{\times}\times \Z_{p}$.
Therefore we can replace $f(x,y)$ by $x^{-m-1}y^n f(x,y)$ and $g(x,y)$ by $x^{-m-1}y^n g(x,y)$.
Then we obtain this proposition.
\end{proof}
Next we show that 
\begin{align}
\int_{\Z_{p}^{\times}\times \Z_{p}}\exp\left(\frac{x\lambda([g^{-1}\bar{\pi}^N u]s_N)}{\Omega_{\mathfrak{p}}}\right)x^{-m-1}y^n
 d\mu_{(\alpha_0 \Omega/g -\gamma\Omega/g)/\bar{\pi}^N, 0}(x,y)
\label{immediate formula of p-adic L-function}
\end{align}
can be expressed as a Coleman function on $E(\C_{p})$.
Note that the value $(\alpha_0 \Omega/g-\gamma\Omega/g)/\bar{\pi}^N$ is an algebraic
element since
$(\alpha_0 \Omega/g-\gamma\Omega/g)/\bar{\pi}^N\in \mathfrak{g}^{-1}\bar{\pi}^{-N}\Gamma/\Gamma\cong E(\C)[\mathfrak{g}\bar{\pi}^N]
\cong E(\overline{K})[\mathfrak{g}\bar{\pi}^N]$,
so that $(\alpha_0 \Omega/g-\gamma\Omega/g)/\bar{\pi}^N$ can be embedded into $E(\C_{p})$ by an inclusion $i:\overline{K}\hookrightarrow \C_{p}$.
For an inclusion map $i_{*}: E(\overline{K})\hookrightarrow E(\C_{p})$ induced by the inclusion map $i$ through the lifting
$
E(\overline{K})[\mathfrak{g}p^N]
\overset{\times \pi^N}{\relbar\joinrel\relbar\joinrel\relbar\joinrel\relbar\joinrel\twoheadrightarrow} 
E(\overline{K})[\mathfrak{g}\overline{\pi}^N],
$
we put $\widetilde{\alpha_0}:=i_{*}(\alpha_0 \Omega/gp^N)$ and $\widetilde{\gamma}:=i_{*}(\gamma \Omega/gp^N)$.
\begin{Lem}
The formula (\ref{immediate formula of p-adic L-function}) can be expressed as a Coleman function on $E(\C_{p})$ as follows:
\begin{align*}
\int_{\Z_{p}^{\times}\times \Z_{p}}&\exp\left(\frac{x\lambda([g^{-1}\bar{\pi}^N u]s_N)}{\Omega_{\mathfrak{p}}}\right)x^{-m-1}y^n
 d\mu_{(\alpha_0 \Omega/g -\gamma\Omega/g)/\bar{\pi}^N, 0}(x,y)&\\
 &=n!(-1)^{m+n+1}\Omega_{\mathfrak{p}}^{n-m-1}E_{m+1,n+1}^{(p)}(\lambda([g^{-1}\overline{\pi}^{-N}u]s_N)+\pi^{N}({\widetilde{\alpha_{0}}}-\widetilde{\gamma})).&
\end{align*}
\label{expressed as a Coleman function}
\end{Lem}
Before proving Lemma \ref{expressed as a Coleman function}, we introduce a new definition.
\begin{Def}[=\cite{BKT} Definition 2.18]
For a non-zero torsion point $z_0 \in E(\overline{\Q})_{\operatorname{tors}}$ with an order to prime $\mathfrak{p}$ and any integer $b\geq 0$, 
we define the $p$-adic measure $\mu_{z_0, b}$ 
on the set $\mathscr{C}^{\operatorname{an}}(\mathcal{O}_{K_{\mathfrak{p}}},\C_{p})$
consisting of locally $K_{\mathfrak{p}}$-analytic functions on $\mathcal{O}_{K_{\mathfrak{p}}}$:
$$
\int_{\mathcal{O}_{K_{\mathfrak{p}}}}\exp\left(\frac{x\lambda(s)}{\Omega_{\mathfrak{p}}}\right)d\mu_{z_0, b}(x):=\widehat{F}_{z_0, b}(s).
$$
\end{Def}
Then according to (\cite{BKT} \S2.4), for any integer $a,b$ with $b\geq 0$, we have
$$
\int_{\Z_{p}^{\times}\times \Z_{p}}x^a y^b d\mu_{(\alpha_0 \Omega/g-\gamma\Omega/g)/\bar{\pi}^N,0}(x,y)=
b!\Omega_{\mathfrak{p}}^{b}\int_{\mathcal{O}_{K_{\mathfrak{p}}}^{\times}}x^a d\mu_{(\alpha_0 \Omega/g-\gamma\Omega/g)/\bar{\pi}^N,b+1}(x)
$$
Since $x^a$ is a continuous function on $\Z_{p}^{\times}\cong \mathcal{O}_{K_{\mathfrak{p}}}^{\times}$ and $a$ is arbitrary, 
$x^a$ is arbitrary continuous function on $\Z_{p}^{\times}\cong \mathcal{O}_{K_{\mathfrak{p}}}^{\times}$.
So we can replace by
$$
x^a\mapsto x^a \exp\left(\frac{x\lambda([g^{-1}\bar{\pi}^{-N}u]s_N)}{\Omega_{\mathfrak{p}}}\right).
$$
If we substitute $a=-m-1, b=n$, then we have
\begin{align}
\label{immediately expressed as a Coleman function}
\int_{\Z_{p}^{\times}\times \Z_{p}}&x^{-m-1} y^n \exp\left(\frac{x\lambda([g^{-1}\bar{\pi}^{-N}u]s_N)}{\Omega_{\mathfrak{p}}}\right)
d\mu_{(\alpha_0 \Omega/g-\gamma\Omega/g)/\bar{\pi}^N,0}(x,y)&\\
&=n!\Omega_{\mathfrak{p}}^{n}\int_{\mathcal{O}_{K_{\mathfrak{p}}}^{\times}} 
x^{-m-1} \exp\left(\frac{x\lambda([g^{-1}\bar{\pi}^{-N}u]s_N)}{\Omega_{\mathfrak{p}}}\right)
d\mu_{(\alpha_0 \Omega/g-\gamma\Omega/g)/\bar{\pi}^N,n+1}(x)\notag
\end{align}
\begin{proof}[\text{\textsc{Proof of  Lemma \ref{expressed as a Coleman function}}}]
For a torsion point $z_0\in E(\overline{K})_{\operatorname{tors}}\overset{i_*}{\hookrightarrow}E(\C_{p})_{\operatorname{tors}}$ with a prime to $\pi$ and 
any integer $b\geq 0$, by using the induction on $m$ and the distribution relation of the $p$-adic Eisenstein-Kronecker series
(\ref{distribution relation of p-adic Eisenstein-Kronecker series}), we can show the following relation (for details, see \cite{BFK} Remark 3.17.):
\begin{equation}
E_{m,b}^{(p)}(z)\Big|_{]i_{*}(z_0)[}
=(-1)^{b-1}(-\Omega_{\mathfrak{p}})^{m}\int_{\mathcal{O}_{K_{\mathfrak{p}}}^{\times}}
x^{-m}\exp\left(\frac{x\lambda(t)}{\Omega_{\mathfrak{p}}}\right)d\mu_{z_0, b}(x)
\label{relation p-adic construction and Coleman function}
\end{equation}
where $t$ and $z$ are variables related by $z=\lambda(t)$ with the normalized formal logarithm $\lambda: \widehat{E}\overset{\cong}{\longrightarrow} \widehat{\G}_{a}$.\\
\quad Now $(\alpha_0 \Omega/g-\gamma\Omega/g)/\bar{\pi}^N$ has an order prime to $\pi$, we can use (\ref{relation p-adic construction and Coleman function}). Then
(\ref{immediately expressed as a Coleman function}) becomes
\begin{align*}
\int_{\Z_{p}^{\times}\times \Z_{p}}x^{-m-1} &y^n \exp\left(\frac{x\lambda([g^{-1}\bar{\pi}^{-N}u]s_N)}{\Omega_{\mathfrak{p}}}\right)
d\mu_{(\alpha_0 \Omega/g-\gamma\Omega/g)/\bar{\pi}^N,0}(x,y)&\\
\intertext{By using (\ref{relation p-adic construction and Coleman function}), we have}
&=n!(-1)^{m+n+1}\Omega_{\mathfrak{p}}^{n-m-1}E_{m+1, n+1}^{(p)}(\lambda([g^{-1}\bar{\pi}^{-N}u]s_N))\Big|_{]\pi^N(\widetilde{\alpha_0}-\widetilde{\gamma})[}&\\
&=n!(-1)^{m+n+1}\Omega_{\mathfrak{p}}^{n-m-1}E_{m+1, n+1}^{(p)}(z)\Big|_{]\pi^N(\widetilde{\alpha_0}-\widetilde{\gamma})[}
\bigg|_{z=\lambda([g^{-1}\bar{\pi}^{-N}u]s_N)}&\\
\intertext{By the translation of the residue disk  (\ref{translation of rigid point}), we have}
&=n!(-1)^{m+n+1}\Omega_{\mathfrak{p}}^{n-m-1} E_{m+1, n+1}^{(p)}(z+\pi^N(\widetilde{\alpha_0}-\widetilde{\gamma}))\Big|_{]0[}
\bigg|_{z=\lambda([g^{-1}\bar{\pi}^{-N}u]s_N)}&
\intertext{Since $z=\lambda([g^{-1}\bar{\pi}^{-N}u]s_N)\in ]0[$, we have}
&=n!(-1)^{m+n+1}\Omega_{\mathfrak{p}}^{n-m-1} E_{m+1, n+1}^{(p)}(\lambda([g^{-1}\bar{\pi}^{-N}u]s_N)+\pi^N(\widetilde{\alpha_0}-\widetilde{\gamma})).&
\end{align*}
Then we obtain this lemma.
\end{proof}
Here for $u\in (\mathcal{O}_{K}/(\pi^N))^{\times}$, the value 
$\lambda([g^{-1}\bar{\pi}^{-N}u]s_N)=g^{-1}\bar{\pi}^{-N}u z_N=u\Omega/gp^N$ defines a primitive $\mathfrak{g}p^N$-torsion point.
Since $u\Omega/gp^N$ is the element in  $E(\overline{K})[\mathfrak{g}p^N]$ through the isomorphism
$\mathfrak{g}^{-1}p^{-N}\Gamma/\Gamma\cong E(\C)[\mathfrak{g}p^N]\cong E(\overline{K})[\mathfrak{g}p^N]$,
the value $u\Omega/gp^N$ can be embedded in $E(\C_{p})$.
For an inclusion map $i_{*}: E(\overline{K})\hookrightarrow E(\C_{p})$ induced by the inclusion map $i:\overline{K}\hookrightarrow \C_{p}$,
we put $\widetilde{u}:=i_{*}(u\Omega/gp^N)$.
Then the value of the $p$-adic $L$-function (\ref{redefinition of p-adic L-function 3}) can be calculated as follows:
\begin{Prop}
Let $m,n$ be any integers with $n\geq 0$. Then we have
\begin{align*}
\frac{L_{p}(\varphi_{p})}{\Omega_{\mathfrak{p}}^{n-m}}=
\frac{g^{-1} n! (-1)^{m+n+1}}{\tau(\overline{\chi_{1}})\overline{\pi}^N}
\sum_{\alpha_0 \in(\mathcal{O}_{K}/\mathfrak{g})^{\times}}\chi_{\mathfrak{g}}(\alpha_0)\sum_{u\in (\mathcal{O}_{K}/(\pi^N))^{\times}}\chi_{1}(u)
\sum_{\gamma \in (\mathcal{O}_{K}/(\bar{\pi}^N))^{\times}}\chi_{2}(\gamma)\times&\\
E_{m+1, n+1}^{\operatorname{col}}(\widetilde{u}+\pi^N(\widetilde{\alpha_0}+\widetilde{\gamma})).\notag
\end{align*} 
\label{redefinition of twisted p-adic L-function}
\end{Prop}
\begin{proof}
First, we combine Lemma \ref{redefinition of immediate p-adic L-function 1} with Lemma \ref{expressed as a Coleman function}.
Second, by Lemma \ref{Properties of Gauss sum} b), we have $\tau(\chi_{2})\tau(\overline{\chi_{2}})=\chi_{2}(-1)p^N$.
Third, we use the fact that $\chi_{2}(-1)^{-1}=\chi_{2}(-1)$ holds since $\chi_{2}(-1)=\pm 1$ and that 
$\gamma\mapsto -\gamma$ is bijective on $(\mathcal{O}_{K}/(\overline{\pi}^N))^{\times}$.
Finally, if we recall (\ref{twisted p-adic Eisenstein-Kronecker series as Coleman function}) and use the following Lemma \ref{cancel}, we obtain this proposition.
\end{proof}
\begin{Lem}
\begin{align*}
\sum_{u\in (\mathcal{O}_{K}/(\pi^N))^{\times}}\chi_{1}(u)
E_{m+1, n+1}^{\operatorname{col}}(\pi(\widetilde{u}+\pi^N(\widetilde{\alpha_0}+\widetilde{\gamma})))=0.
\end{align*}
\label{cancel}
\end{Lem}
\begin{proof}
Since $(\pi, \pi^N)=\pi$ and $\chi_{1}$ is a primitive character on $(\mathcal{O}_{K}/(\pi^N))^{\times}$, there exists
$a\in (\mathcal{O}_{K}/(\pi^N))^{\times}$ such that
$$
\chi_{1}(a)\neq 1\quad \text{and}\quad a\equiv 1\ (\operatorname{mod}\ \pi^{N-1}).
$$
Therefore
\begin{align*}
\chi_{1}(a)&\sum_{u\in (\mathcal{O}_{K}/(\pi^N))^{\times}}\chi_{1}(u)
E_{m+1, n+1}^{\operatorname{col}}(\pi(\widetilde{u}+\pi^N(\widetilde{\alpha_0}+\widetilde{\gamma})))&\\
&=\sum_{u\in (\mathcal{O}_{K}/(\pi^N))^{\times}}\chi_{1}(au)
E_{m+1, n+1}^{\operatorname{col}}(\pi(\widetilde{u}+\pi^N(\widetilde{\alpha_0}+\widetilde{\gamma})))&\\
\intertext{since $a$ is prime to $\pi$, $u\mapsto a^{-1}u$ is bijective on $(\mathcal{O}_{K}/(\pi^N))^{\times}$, then we have}
&=\sum_{u\in (\mathcal{O}_{K}/(\pi^N))^{\times}}\chi_{1}(u)
E_{m+1, n+1}^{\operatorname{col}}(\pi(a^{-1}\widetilde{u}+\pi^N(\widetilde{\alpha_0}+\widetilde{\gamma})))&\\
\intertext{since $a\equiv 1\ (\operatorname{mod} \pi^{N-1})$ and the value  $a^{-1}\widetilde{u}=i_{*}(\alpha^{-1}u\Omega/gp^N)$ 
is the image of $\mathfrak{g}p^N$-torsion point in 
$\mathfrak{g}^{-1}p^{-N}\Gamma/\Gamma\cong E(\C)[\mathfrak{g}p^N]$, we have}
&=\sum_{u\in (\mathcal{O}_{K}/(\pi^N))^{\times}}\chi_{1}(u)
E_{m+1, n+1}^{\operatorname{col}}(\pi(\widetilde{u}+\pi^N(\widetilde{\alpha_0}+\widetilde{\gamma}))).&
\end{align*}
Since $\chi_{1}(a)\neq 1$, this lemma holds.
\end{proof}
We order Proposition \ref{redefinition of twisted p-adic L-function} by gathering toward a character on $(\mathcal{O}_{K}/\mathfrak{g}')^{\times}$.
First we gather $\chi_{\mathfrak{g}}, \chi_{1}, \chi_{2}$ by solving the following simultaneous congruences of first degree:
$$
z\equiv \alpha_{0}\ (\operatorname{mod} \mathfrak{g}),\  z\equiv u\ (\operatorname{mod} \pi^N),\
z\equiv \gamma\ (\operatorname{mod} \overline{\pi}^N). 
$$
We assume that $z=s_1\in \mathcal{O}_{K}$ is a solution of $p^N z=\pi^N \overline{\pi}^N z\equiv 1(\operatorname{mod} \mathfrak{g})$.
The existence of this solution follows from that $\mathfrak{g}$ is prime to $p$. Similarly, we assume that $z=s_2\in \mathcal{O}_{K}$ (resp. $z=s_3\in \mathcal{O}_{K}$)
is a solution of $g\overline{\pi}^N z\equiv 1 (\operatorname{mod} \pi^N)$ (resp. $g \pi^N z\equiv 1 (\operatorname{mod} \overline{\pi}^N)$).
If we put $z=\alpha_0 p^N s_1+ug\overline{\pi}^N s_2+\gamma g\pi^N s_3$, 
$z$ satisfies simultaneously $z\equiv \alpha_{0}\ (\operatorname{mod} \mathfrak{g}),\  z\equiv u\ (\operatorname{mod} \pi^N),\
z\equiv \gamma\ (\operatorname{mod} \overline{\pi}^N)$. Therefore the solution of the given simultaneous congruences of first degree is
$$
z\equiv \alpha_0 p^N s_1+ug\overline{\pi}^N s_2+\gamma g\pi^N s_3\ (\operatorname{mod} \mathfrak{g}p^N).
$$
Recalling $\chi=\chi_{\mathfrak{g}}\chi_{1}\chi_{2}$ and $\mathfrak{g}'=\mathfrak{g}p^N$, Proposition \ref{redefinition of twisted p-adic L-function} is
\begin{align*}
\frac{L_{p}(\varphi_{p})}{\Omega_{\mathfrak{p}}^{n-m}}=
\frac{g^{-1} n! (-1)^{m+n+1}}{\tau(\overline{\chi_{1}})\overline{\pi}^N}
\sum_{z\in (\mathcal{O}_{K}/\mathfrak{g}')^{\times}}\chi(z)
E_{m+1, n+1}^{\operatorname{col}}(\widetilde{u}+\pi^N(\widetilde{\alpha_0}+\widetilde{\gamma})).
\end{align*}
Next, we express the contents $\widetilde{u}+\pi^N(\widetilde{\alpha_0}+\widetilde{\gamma})$ of $E_{m+1, n+1}^{\operatorname{col}}$ by the formula of $z$.
We find a constant $C$ satisfying
$$
u+\pi^N (\alpha_0+\gamma)=Cz=C(\alpha_0 p^N s_1+ug\overline{\pi}^N s_2+\gamma g\pi^N s_3)
$$
in $\mathcal{O}_{K}/\mathfrak{g}'$. Since $u$, $\alpha_0$, and $\gamma$ are arbitrary, comparing respectively coefficients of $u$, $\alpha_0$, and $\gamma$,
we have
$$
1=Cg\overline{\pi}^N s_2,\quad \pi^N=Cp^Ns_1,\quad \pi^N=Cg\pi^N s_3
$$
in $\mathcal{O}_{K}/\mathfrak{g}'$.
Since $s_1$, $s_2$, and $s_3$ respectively satisfies $p^N s_1\equiv 1(\operatorname{mod} \mathfrak{g})$,
 $g\overline{\pi}^N s_2\equiv 1 (\operatorname{mod} \pi^N)$, and 
 $g \pi^N s_3\equiv 1 (\operatorname{mod} \overline{\pi}^N)$, we have the following simultaneous congruences of first degree:
$$
C\equiv 1(\operatorname{mod} \pi^N),\quad C\equiv \pi^N (\operatorname{mod}\mathfrak{g}), \quad C\equiv \pi^N (\operatorname{mod}\overline{\pi}^N).
$$
Solving these simultaneous congruences of first degree, for the above $s_1, s_2$, and $s_3$, we have
$$
C\equiv \pi^N p^N s_{1}+g\overline{\pi}^N s_{2}+g\pi^{2N}s_{3}(\operatorname{mod}\mathfrak{g}p^N).
$$
Now if we put $\xi_{\mathfrak{g}'}:=i_{*}(\Omega( \pi^N p^N s_{1}+g\overline{\pi}^N s_{2}+g\pi^{2N}s_{3})/gp^N)=i_{*}(\Omega C/gp^N)$,
$\xi_{\mathfrak{g}'}$ is a primitive $\mathfrak{g}'$-torsion point.
$\xi_{\mathfrak{g}'}$ is the primitive $\mathfrak{g}'$-torsion point in $E(\C_{p})$, which is the special case 
when we substitute $\alpha_0=\pi^N$, $u=1$, and $\gamma=\pi^N$
for $z\equiv \alpha_0 p^N s_1+ug\overline{\pi}^N s_2+\gamma g\pi^N s_3 (\operatorname{mod} \mathfrak{g}p^N)$.
Then
we have
\begin{Thm}[Main Theorem]
Let $m,n$ be any integers with $n\geq 0$.
For the above symbols, 
we have
\begin{align*}
\frac{L_{p}(\varphi_{p})}{\Omega_{\mathfrak{p}}^{n-m}}=
\frac{g^{-1} n! (-1)^{m+n+1}}{\tau(\overline{\chi_{1}})\overline{\pi}^N}
\sum_{z\in (\mathcal{O}_{K}/\mathfrak{g}')^{\times}}\chi(z)
E_{m+1, n+1}^{\operatorname{col}}(\xi_{\mathfrak{g}'} z),
\end{align*}
where $\tau(\overline{\chi_{1}})$ is the Gauss sum defined by Lemma \ref{Properties of Gauss sum}.
\label{My Main Theorem}
\end{Thm}
Main Theorem is the $p$-adic analogue of the relation of the complex Hecke $L$-function and the classical Eisenstein-Kronecker series
\begin{equation}
L_{\mathfrak{g}'}(s, \varphi)=\frac{1}{w_{\mathfrak{g}'}}K_{|m-n|}^{*}(\alpha, 0, s-\min\{m,n \}; ((\alpha)^{-1}\mathfrak{g}')^{\delta}),
\label{Hecke L-function and Eisenstein-Kronecker series}
\end{equation}
where $w_{\mathfrak{g}'}$ is the number of roots of 1 in $\mathcal{O}_{K}^{\times}$ congruent to 1 modulo $\mathfrak{g}'$,
$\alpha$ is any element of $\mathfrak{a}^{-1}$ in (\ref{The classical complex Hecke L-function}) such that $\alpha\equiv 1\ (\operatorname{mod} \mathfrak{g}')$ and
$\delta\in \operatorname{Gal}(\C/\R)$ is an element satisfying that $\delta$ is trivial if and only if $m-n>0$ (see \cite{BK} Proposition 1.6).

\section*{Acknowledgements}
The result of this article is an extended version of the result of my master's thesis \cite{H}.
This research was supported by KAKENHI 21674001. 
Thanks to this KAKENHI, I could promote this research smoothly.
I wish to thank my professor K. Bannai and other post doctors for having sincerely advised and led when I was in trouble.

\end{document}